\documentclass[a4paper, 12pt]{article}
\usepackage[utf8]{inputenc}
\usepackage[english]{babel}
\usepackage{amssymb,amsmath,amsthm,eucal}
\usepackage[left=1cm, right=1cm, top=1cm, bottom=2cm]{geometry}

\usepackage{hyperref}
\hypersetup{
    colorlinks=true,
    linkcolor=blue,
    filecolor=magenta,      
    urlcolor=cyan,
}
\title{Gelfand-Tsetlin degeneration of shift of argument subalgebras in types B and C}
\author{Leonid Rybnikov and Mikhail Zavalin}

\begin{document}

\newtheorem{fact}{Fact}
\newtheorem{predl}{Theorem}
\newtheorem{prop}{Proposition}
\newtheorem{lemma}{Lemma}
\newtheorem{conj}{Conjecture}
\newtheorem{cor}{Corollary}
\newtheorem{rem}{Remark}
\newtheorem{defn}{Definition}

\maketitle
\begin{center}
{\it To our teacher Rafail Kalmanovich Gordin}
\end{center}

\medskip

\begin{abstract} The universal enveloping algebra of any semisimple Lie algebra $\mathfrak{g}$ contains a family of maximal commutative subalgebras, called shift of argument subalgebras, parametrized by regular Cartan elements of $\mathfrak{g}$. For $\mathfrak{g}=\mathfrak{gl}_n$ the Gelfand-Tsetlin commutative subalgebra in $U(\mathfrak{g})$ arises as some limit of subalgebras from this family. We study the analogous limit of shift of argument subalgebras for the Lie algebras $\mathfrak{g}=\mathfrak{sp}_{2n}$ and $\mathfrak{g}=\mathfrak{so}_{2n+1}$. The limit subalgebra is described explicitly in terms of Bethe subalgebras in twisted Yangians $Y^-(2)$ and $Y^+(2)$, respectively. We index the eigenbasis of such limit subalgebra in any irreducible finite-dimensional representation of $\mathfrak{g}$ by Gelfand-Tsetlin patterns of the corresponding type, and conjecture that this indexing is, in appropriate sense, natural. According to \cite{KR} such eigenbasis has a natural $\mathfrak{g}$-crystal structure. We conjecture that this crystal structure coincides with that on Gelfand-Tsetlin patterns defined by Littelmann in \cite{L}.  
\end{abstract}

\tableofcontents
\newpage

\addcontentsline{toc}{section}{Introduction}
\section{Introduction}
\subsection{Maximal commutative subalgebras in $S(\mathfrak{g})$ and $U(\mathfrak{g})$.}

\textit{Shift of argument subalgerbas} form a family of maximal Poisson-commutative subalgebras of the Poisson algebra $S(\mathfrak{g})$ of a semisimple Lie algebra $\mathfrak{g}$. These subalgebras are parametrized by regular elements $\mu\in\mathfrak{g}^*$. More precisely, for any $\mu\in\mathfrak{g}^*$ all partial derivatives of $\mathfrak{g}$-invariants in $S(\mathfrak{g})$ along $\mu$ generate a Poisson-commutative subalgebra $A_{\mu}\subset S(\mathfrak{g})$. For regular $\mu\in\mathfrak{g}^*$ the subalgebra $A_\mu$ is known to be a polynomial algebra in $\frac{1}{2}(\dim\mathfrak{g}+\text{rk}\mathfrak{g})$ generators (hence having maximal possible transcendence degree).  These subalgebras were first introduced by Mishchenko and Fomenko in \cite{MF} and are also known as \emph{Mishchenko-Fomenko subalgebras}. 

We fix an invariant scalar product on $\mathfrak{g}$ (thus having $\mathfrak{g}=\mathfrak{g}^*$) and a Cartan decomposition $\mathfrak{g}=\mathfrak{n}_+\oplus\mathfrak{h}\oplus\mathfrak{n}_-$. From now on we assume that the parameter $\mu$ is a regular element from $\mathfrak{h}\subset\mathfrak{g}=\mathfrak{g}^*$. Note that the subalgebra $A_\mu\subset S(\mathfrak{g})$ does not change under dilations of $\mu$, so the parameter space for the shift of argument subalgebras with $\mu\in\mathfrak{h}$ is the projectivization of the set of regular Cartan elements, $\mathbb{P}(\mathfrak{h})^{reg}$. 

The space $\mathbb{P}(\mathfrak{h})^{reg}$, which parametrizes the family $A_\mu$, is noncompact. Following \cite{Sh,V1} we extend this family of subalgebras to some compactification of $\mathbb{P}(\mathfrak{h})^{reg}$. Namely, one can consider \emph{limit} shift of argument subalgebras obtained as the $A_{\mu(\varepsilon)}$ as $\varepsilon\to0$ where element $\mu(\varepsilon)\in\mathfrak{g}$ is regular and semisimple for sufficiently small values of $\varepsilon$. These subalgebras are also known to be polynomial algebras in $\frac{1}{2}(\dim\mathfrak{g}+\text{rk}\ \mathfrak{g})$ generators, see \cite{Sh}. In \cite{KR} the parameter space for all possible limit subalgebras was identified with the De Concini-Procesi wonderful closure of $\mathbb{P}(\mathfrak{h})^{reg}$ (regarded as the complement of the root hyperplane arrangement).

In \cite{V1} Vinberg raised the problem of quantization of subalgebras $A_{\mu}$, i.e. lifting them to commutative subalgebras $\mathcal{A}_\mu$ in the universal enveloping algebra $U(\mathfrak{g})$ such that $\text{gr}\ \mathcal{A}_\mu=A_{\mu}$. In \cite{R} this problem was solved for regular $\mu$. In \cite{R2} it was shown that this lifting is unique for generic $\mu$. Moreover, according to \cite{KR} this lifting extends uniquely to the limit subalgebras. 

\subsection{Gelfand-Tsetlin limit of shift of argument subalgebras for $\mathfrak{g}=\mathfrak{gl}_n$.}

Let $\mathfrak{g}=\mathfrak{gl}_n$ and let $E_{ij}$ ($i,j=1,...,n$) denote matrix units of $\mathfrak{gl}_n$. In \cite{V1} Vinberg observed that the limit shift of argument subalgebra $\lim\limits_{\varepsilon\to0}A_{\mu(\varepsilon)}$ for $\mu(\varepsilon)=E_{nn}+E_{n-1,n-1}\varepsilon+...+E_{11}\varepsilon^{n-1}$ is the associated graded of the \textit{Gelfand-Tsetlin subalgebra} $\mathcal{G}_n$ of $U(\mathfrak{gl}_n)$. This subalgebra has simple spectrum in any irreducible $\mathfrak{gl}_n$-module and is diagonalizable in \textit{Gelfand-Tsetlin basis}. Here we describe these objects following \cite{M}.

For irreducible $\mathfrak{gl}_n$-module $V_\lambda$ with the highest weight $\lambda=(\lambda_1,\lambda_2,...,\lambda_n)$ its restriction to $\mathfrak{gl}_{n-1}$ is isomorphic to the direct sum of $V'_\mu$ over all $\mathfrak{gl}_{n-1}$-weights $\mu$, satisfying $\lambda_i-\mu_i\in\mathbb{Z}_+$ and $\mu_i-\lambda_{i+1}\in\mathbb{Z}_{+}$. Iterating this restriction for the chain of embedded Lie subalgebras $\mathfrak{gl}_n\supset\mathfrak{gl}_{n-1}\supset\ldots\supset\mathfrak{gl}_1$ we get a basis of $V_\lambda$, indexed by \textit{Gelfand-Tsetlin patterns} $\Lambda$
$$\begin{matrix}
\lambda_{n1} &\lambda_{n2} &&& \ldots &&& \lambda_{nn}\\
&\lambda_{n-1,1} & \lambda_{n-1,2} && \ldots && \lambda_{n-1,n-1} &&\\
&&\ldots && \ldots && \ldots &&\\
&&&\lambda_{21} && \lambda_{22} &&\\
&&&&\lambda_{11} &&&&\\
\end{matrix}$$
with $\lambda_i$ satisfying conditions $\lambda_{ij}-\lambda_{i-1,j}\in\mathbb{Z}_+$
and $\lambda_{i-1,j}-\lambda_{i,j+1}\in\mathbb{Z}_+$. This basis $\xi_{\Lambda}$ is called the \textit{Gelfand-Tsetlin basis} of $V_\lambda$. 

The Gelfand-Tsetlin subalgebra $\mathcal{G}_n\subset U(\mathfrak{gl}_n)$ is generated by the centers of $U(\mathfrak{gl}_k)$, $k=1,\ldots,n$. Clearly all elements from $\mathcal{G}_n$ are diagonal in the Gelfand-Tsetlin basis. The eigenvalues of central generators of $U(\mathfrak{gl}_k)$ are shifted elementary symmetric functions of the elements of the $k$-th row in the Gelfand-Tsetlin pattern described above. In particular, the joint eigenvalues on different elements of the Gelfand-Tsetlin basis are different. So $\mathcal{G}_n$ has simple spectrum in any irreducible representation of $\mathfrak{gl}_n$.

\subsection{The cases of $\mathfrak{g}=\mathfrak{sp}_{2n}$ and $\mathfrak{g}=\mathfrak{o}_{2n+1}$.}

Recall that for $N=2n$ the following elements of $\mathfrak{gl}_N$
$$F_{ij}:=E_{ij}-\theta_{ij}E_{-j,-i},$$
where $i,j\in\{-n,...,-1,1,...,n\}$ for $\theta_{ij}:=\textrm{sgn}(i)\textrm{sgn}(j)$ span the subalgebra $\mathfrak{g}_n=\mathfrak{sp}_{2n}=\mathfrak{sp}_N\subset\mathfrak{gl}_N$. For $N=2n+1$, the above elements with $i,j\in\{-n,...,-1,0,1,...,n\}$ and $\theta_{ij}:=1$ span the subalgebra $\mathfrak{g}_n=\mathfrak{o}_{2n+1}=\mathfrak{o}_N\subset\mathfrak{gl}_N$.

In this paper we describe the limit of the quantum shift of argument subalgebra, $\lim\limits_{\varepsilon\to0}\mathcal{A}_{\mu(\varepsilon)}$ for $\mu(\varepsilon)=F_{nn}+F_{n-1,n-1}\varepsilon+...+F_{11}\varepsilon^{n-1}$. As in the $\mathfrak{gl}_n$ case, we consider the chain of Lie subalgebras $\mathfrak{g}_{n}\supset \mathfrak{g}_{n-1}\supset\ldots\supset \mathfrak{g}_{1}$ embedded in the standard way (i.e. $\mathfrak{g}_k$ is generated by $F_{ij}$ with $-k\le i,j\le k$). The problem is that contrary to the case of $\mathfrak{gl}_n$, the centers of all $U(\mathfrak{g}_{k})$ from this chain generate a smaller subalgebra than is expected from the limiting procedure. In particular, the joint spectrum of such centers is not simple. It turns out that in both symplectic and orthogonal cases we have some additional generators of this limit subalgebra. To describe them, we consider the \emph{twisted Yangian} $\textrm{Y}^{-}(2)$ (resp. $\textrm{Y}^{+}(2)$) and its commutative Bethe subalgebra $\mathcal{B}^{-}$ (resp. $\mathcal{B}^{+}$) for the case of $\mathfrak{sp}_{2n}$ (resp. $\mathfrak{o}_{2n+1}$). By \cite{MO}, for both symplectic and orthogonal cases, there is a homomorphism $$\varphi_k:  \textrm{Y}^{\mp}(2)\to  U(\mathfrak{g}_{k})^{\mathfrak{g}_{k-1}},$$ which is surjective modulo the center of $U(\mathfrak{g}_{k})$. We denote by $\mathcal{A}^{\mp}_k$ the subalgebra in $U(\mathfrak{g}_{k})^{\mathfrak{g}_{k-1}}$ generated by the image of commutative subalgebra $\mathcal{B}^{\mp}$ under this homomorphism and the center of $U(\mathfrak{g}_{k})$. The subalgebra $\mathcal{A}^{\mp}\subset U(\mathfrak{g}_n)$ generated by the union $\cup_{k=1}^n\mathcal{A}^{\mp}_k$ is a commutative subalgebra of $U(\mathfrak{g}_n)$. The first main result of the present paper is the following

\medskip

\noindent\textbf{Theorem A.} \textit{The limit of the quantum shift of argument subalgebras $\lim\limits_{\varepsilon\to0}\mathcal{A}_{\mu(\varepsilon)}$ for $\mu(\varepsilon)=F_{nn}+F_{n-1,n-1}\varepsilon+...+F_{11}\varepsilon^{n-1}$ in the universal enveloping algebra $U(\mathfrak{g}_n)$ coincides with $\mathcal{A}^{-}$ for $\mathfrak{g}_n=\mathfrak{sp}_{2n}$ and with $\mathcal{A}^{+}$ for $\mathfrak{g}_n=\mathfrak{o}_{2n+1}$.}

\medskip

\begin{rem} \emph{The case of $\mathfrak{g}_n=\mathfrak{o}_{2n}$ seems harder due to Pfaffian-related reasons. We are going to consider this case separately in our future work.}
\end{rem}

\begin{rem} \emph{In the case of $\mathfrak{g}_n=\mathfrak{sp}_{2n}$ the subalgebra $\lim\limits_{\varepsilon\to0}\mathcal{A}_{\mu(\varepsilon)}$ was described by Molev and Yakimova in \cite{MY}. We show that $\mathcal{A^-}$ is the same subalgebra.}
\end{rem}

Now we want to describe the spectra of $\mathcal{A}^{\mp}$ in irreducible finite-dimensional representations of $\mathfrak{g}_n$. According to \cite{KR}, for any dominant integral highest weight $\lambda$ the spectrum of any limit shift of argument subalgebra in the corresponding representation $V_\lambda$ is simple. On the other hand, there are no explicit formulas for the joint eigenvalues of Bethe subalgebras in twisted Yangians known. The eigenvectors can be obtained by an appropriate version of Bethe ansatz method, see \cite{ansatz}. We do not address Bethe ansatz completeness problem in this paper. Rather, for the cases $\mathfrak{g}_n=\mathfrak{sp}_{2n}$ and $\mathfrak{g}_n=\mathfrak{o}_{2n+1}$ we describe some indexing of the eigenbasis of $\mathcal{A}^\mp$ in $V_\lambda$ by the analogs of Gelfand-Tsetlin patterns. This indexing depends on the choice of a path in a certain parameter space. We conjecture that this indexing is natural, i.e. does not in fact depend on any choices. 

\subsection{Gelfand-Tsetlin type patterns for Lie algebras $\mathfrak{sp}_{2n}$ and $\mathfrak{o}_{2n+1}$.}\label{ss:GTpatterns}
In \cite{M} Molev describes the uniform way to get a Gelfand-Tsetlin type bases for all classical Lie algebras, based on the representation theory of twisted Yangians.

For $\mathfrak{g}=\mathfrak{sp}_{2n}$ Gelfand-Tsetlin basis of irreducible finite-dimensional representation does not exist because the restriction $\mathfrak{sp}_{2n} \downarrow \mathfrak{sp}_{2n-2}$ is not multiplicity-free. Instead, there is an action of the twisted Yangian $\textrm{Y}^{-}(2)$ on $U(\mathfrak{sp}_{2n})^{\mathfrak{sp}_{2n-2}}$ through homomorphism $\varphi_n$ providing a $\textrm{Y}^-(2)$-module structure on each multiplicity space. The multiplicity spaces turn out to be irreducible highest weight $\textrm{Y}^-(2)$-modules thus leading to the generalization of the Gelfand-Tsetlin basis called \textit{weight basis}. The elements of the weight basis of an irreducible representation $V_\lambda$ of highest weight $\lambda$ are numbered by the following combinatorial objects, called \textit{C type patterns} $\Lambda$
$$\begin{matrix}
& \lambda_{n1} && \lambda_{n2} && \ldots && \lambda_{nn}\\
\lambda'_{n1} && \lambda'_{n2} && \ldots  && \lambda'_{nn}\\
& \lambda_{n-1,1} && \ldots && \lambda_{n-1,n-1}\\
\lambda'_{n-1,1} && \ldots && \lambda'_{n-1,n-1}\\
&& \ldots && \ldots\\
& \lambda_{11}\\
\lambda'_{11}\\
\end{matrix}$$
with $\lambda=(\lambda_{n1},...,\lambda_{nn})$ being the highest weight (which for $\mathfrak{sp}_{2n}$ satisfy $-\lambda_1\in\mathbb{Z}_{+}$ and $\lambda_i-\lambda_{i+1}\in\mathbb{Z}_+$) and the rest entries being non-positive integers, satisfying the following inequalities:
$$0\geq\lambda'_{k1}\geq\lambda_{k1}\geq\lambda'_{k2}\geq\lambda_{k2}\geq...\geq\lambda'_{k,k-1}\geq\lambda_{k,k-1}\geq\lambda'_{kk}\geq\lambda_{kk}$$
for $k=1,...,n$, and
$$0\geq\lambda'_{k1}\geq\lambda_{k-1,1}\geq\lambda'_{k2}\geq\lambda_{k-1,2}\geq...\geq\lambda'_{k,k-1}\geq\lambda_{k-1,k-1}\geq\lambda'_{kk}$$
for $k=2,...,n$.

In the orthogonal case we similarly have an irreducible highest weight $\textrm{Y}^+(2)$-module structure on multiplicity spaces of irreducible highest weight finite-dimensional representation $V_\lambda$ of $\mathfrak{o}_N$ restricted to $\mathfrak{o}_{N-2}$. 
For $\mathfrak{g}=\mathfrak{o}_{2n+1}$, this gives a weight basis of irreducible representation $V_\lambda$ with elements numbered by \textit{B type patterns} $\Lambda$
$$\begin{matrix}
\sigma_n && \lambda_{n1} && \lambda_{n2} && \ldots && \lambda_{nn}\\
& \lambda'_{n1} && \lambda'_{n2} && \ldots  && \lambda'_{nn}\\
\sigma_{n-1} && \lambda_{n-1,1} && \ldots && \lambda_{n-1,n-1}\\
& \lambda'_{n-1,1} && \ldots && \lambda'_{n-1,n-1}\\
& \ldots & \ldots && \ldots &\\
\sigma_1 && \lambda_{11}\\
& \lambda'_{11}
\end{matrix}$$
with $\lambda=(\lambda_{n1},...,\lambda_{nn})$,  $\sigma_i\in\{0,1\}$ and all $\lambda$-entries are simultaneously from $\mathbb{Z}_{\leq0}$ or $\{m+\frac{1}{2}|m\in\mathbb{Z},m+\frac{1}{2}\leq0\}$, satisfying the following set of inequalities
$$\lambda'_{k1}\geq\lambda_{k1}\geq\lambda'_{k2}\geq\lambda_{k2}\geq...\geq\lambda'_{k,k-1}\geq\lambda_{k,k-1}\geq\lambda'_{kk}\geq\lambda_{kk}$$
for $k=1,...,n$, and
$$\lambda'_{k1}\geq\lambda_{k-1,1}\geq\lambda'_{k2}\geq\lambda_{k-1,2}\geq...\geq\lambda'_{k,k-1}\geq\lambda_{k-1,k-1}\geq\lambda'_{kk}$$
for $k=2,...,n$. Additionally, if $\lambda$ consists of integers and $\sigma_k=1$ then we have $\lambda'_{k1}\leq-1$.

\subsection{Indexing the eigenvectors by patterns.}
Suppose that $\mathfrak{g}=\mathfrak{g}_n$ is either $\mathfrak{sp}_{2n}$ or $\mathfrak{o}_{2n+1}$. According to \cite{KR} the spectrum of limit algebra $\mathcal{A}^\mp=\lim\limits_{\varepsilon\to0}\mathcal{A}_{\mu(\varepsilon)}$ in any irreducible finite-dimensional $\mathfrak{g}$-module $V_\lambda$ is simple. Moreover, from Theorem~A we know that the limit algebra $\mathcal{A}^\mp$ is generated by commutative subalgebras $\mathcal{A}_k=\varphi_k(\mathcal{B}^{\mp})$ in the successive centralizer algebras $U(\mathfrak{g}_k)^{\mathfrak{g}_{k-1}}$ (which are quotients of $\text{Y}^\mp(2)$). So the eigenbasis of $V_\lambda$  with respect to $\mathcal{A}^\mp$ agrees with the decomposition of $V_\lambda=\bigoplus\limits_{\mu} V_\lambda^\mu\otimes V_\mu$ with respect to $\mathfrak{g}_{n-1}$. This means that any eigenvector has the form $v=u\otimes w$ where $u\in V_\lambda^\mu$ and $w\in V_\mu$ for some $\mu$. Applying the same argument to the factor $w$ and proceeding by induction we obtain that any eigenvector $v$ with respect to $\mathcal{A}^{\mp}$ in $V_\lambda$ has the following form: for some collection of highest weights $\lambda_k$ of $\mathfrak{g}_k$ (with $\lambda_n=\lambda$) have $v=\bigotimes\limits_{k=1}^nu_k$ where $u_k$ is an eigenvector of $\mathcal{B}^{\mp}$ in the multiplicity space $V^{\lambda_{k-1}}_{\lambda_k}$. 

The multiplicity space $V^{\lambda_{k-1}}_{\lambda_k}=\text{Hom}_{\mathfrak{g}_{k-1}}(V_{\lambda_{k-1}},V_{\lambda_k})$ is known to be the (sum of) tensor products $\bigotimes\limits_{j=1}^kL(\alpha_{kj},\beta_{kj})\otimes W$ where $L(\alpha_{kj},\beta_{kj})$ is the restriction of the ``string'' representation of $Y(2)$ to $Y^{\mp}(2)$, and $W$ is a $1$-dimensional representation of $Y^{\mp}(2)$. The restriction of each factor $L(\alpha_{kj},\beta_{kj})$ to $\mathfrak{sl}_2\subset Y(2)$ is just the irreducible module with the highest weight $\alpha_{kj}-\beta_{kj}$. So we can regard the eigenbasis for $\mathcal{A}^\mp$ in $V_\lambda$ as an element of the continuous family of eigenbases for $\bigotimes\limits_{k=1}^n\mathcal{B}^\mp$ in the tensor products $\bigotimes\limits_{k=1}^n\bigotimes\limits_{j=1}^kL(\alpha_{kj},\beta_{kj})$ with $\alpha_{kj}$ being free parameters and the differences $\alpha_{kj}-\beta_{kj}$ being fixed integers. The second main result of the present paper is the following  

\medskip
\noindent\textbf{Theorem B.} \textit{There is a path $\alpha(t)$ ($t\in[0,\infty)$) in the space of parameters $\alpha_{kj}$ such that \begin{itemize} \item $\alpha(0)$ is the collection of $\alpha_{kj}$ which occurs in our $\mathfrak{g}_n$-module $V_\lambda$;
\item the spectrum of $\bigotimes\limits_{k=1}^n\mathcal{B}^{\mp}$ on $\bigoplus\bigotimes\limits_{k=1}^n\bigotimes\limits_{j=1}^kL(\alpha_{kj}(t),\beta_{kj}(t))\otimes W$ is simple for all $t>0$;
\item the limit of the corresponding eigenbasis as $t\to+\infty$ is just the product of the $\mathfrak{sl}_2$-weight bases in each $L(\alpha_{kj},\beta_{kj})=V_{\alpha_{kj}-\beta_{kj}}$.
\end{itemize}}
\medskip

The weight basis of $V_{\alpha_{kj}-\beta_{kj}}$ is numbered by the integers $\lambda_{kj}'$ satisfying the betweenness conditions from the Gelfand-Tsetlin pattern (of the corresponding type). So, as a corollary of Theorem~B, we get an indexing of the eigenbasis of $\mathcal{A}^\mp$ in our $\mathfrak{g}_n$-module $V_\lambda$ by the Gelfand-Tsetlin patterns described in Section~\ref{ss:GTpatterns}. By construction, our indexing of the eigenbasis depends on the choice of the path $\alpha(t)$. We conjecture that this indexing is in fact independent on this choice.

\subsection{Relation to crystals.} According to \cite{KR} there is a natural $\mathfrak{g}_n$-crystal structure on the set of eigenlines for any limit subalgebra from the family $\mathcal{A}_\mu$ in the space $V_\lambda$. In particular we have such a structure on the set of eigenlines for $\mathcal{A}^\mp$ acting on $V_\lambda$ which is the set of Gelfand-Tsetlin patterns according to Theorem~B. On the other hand, in \cite{L} Littelmann defines a crystal structure on the set Gelfand-Tsetlin patterns for all classical types. We conjecture that these two crystal structures on Gelfand-Tsetlin patterns are the same.


\subsection{The paper is organized as follows.} In Section~1 we recall some classical facts about twisted Yangians $\textrm{Y}^{\mp}(N)$ and their Bethe subalgebras following \cite{MNO} and \cite{NO}. In Section~2 we discuss shift of argument subalgebras, their limits and their quantization for regular semisimple $\mu\in\mathfrak{g}^*$ following \cite{V1}, \cite{Sh} and \cite{R}. Section~3 is devoted to the proof Theorem~A. In Section~4 we prove Theorem~B and formulate conjectures relating it to Littelmann's presentation of crystals.

\subsection{Acknowledgements.} We thank Alexander Molev for explanations on Yangians. The paper was finished during our stay at University of Tokyo. We are grateful to University of Tokyo and especially to Junichi Shiraishi for hospitality. The work of both authors has been funded by the Russian Academic Excellence Project ’5-100’. Theorem~B has been obtained under support of the Russian Science Foundation grant 16-11-10160. The first
author has also been supported in part by the Simons Foundation.

\newpage

\section{Yangians.}
\subsection{Yangian $\textrm{Y}(N)$.}

\begin{defn} \emph{The \emph{Yangian} $Y(N)$ is the associative unital algebra with infinite family of generators $t^{(r)}_{ij}$ with $i,j=1,...,N$ and $r\in\mathbb{Z}_{\geq0}$ satisfying the following conditions:
\begin{equation}\label{sootnY(N)}
[t^{(r+1)}_{ij}, t^{(s)}_{kl}]-[t^{(r)}_{ij}, t^{(s+1)}_{kl}]=t^{(r)}_{kj}t^{(s)}_{il}-t^{(s)}_{kj}t^{(r)}_{il},
\end{equation}
with $i,j,k,l=1,...,N$; $r,s\in\mathbb{Z}_{\geq0}$ and $t^{(0)}_{ij}:=\delta_{ij}.$}
\end{defn}

For $i,j=1,...,N$ introduce the power series $t_{ij}(u)$ in $u^{-1}$: 
\begin{equation}\label{T-series}
t_{ij}(u):=t^{(0)}_{ij}+t^{(1)}_{ij}u^{-1}+t^{(2)}_{ij}u^{-2}+...=\sum_{r\geq0}t_{ij}^{(r)}u^{-r}\in\textrm{Y}(N)\otimes\mathbb{C}[[u^{-1}]],
\end{equation}
and unite them all in the following \emph{T-matrix}:
\begin{equation}\label{T-matrix}
T(u):=\sum^N_{i,j=1}t_{ij}(u)\otimes E_{ij}\in\textrm{Y}(N)\otimes\textrm{End}(W)
\end{equation}
with $W=\mathbb{C}^N$. The defining relations then can be written as a single relation in the algebra $\textrm{Y}(N)\otimes\textrm{End}(W)^{\otimes2}((u^{-1},v^{-1}))$: 
\begin{equation}\label{trenary_relation}
R_{12}(u-v)T_1(u)T_2(v)=T_2(v)T_1(u)R_{12}(u-v),
\end{equation}
where $R_{12}(u-v)=1-\frac{\sum_{i,j=1}^NE_{ij}\otimes E_{ji}}{u-v}\in 1\otimes\textrm{End}(W)^{\otimes2}$ and the subscript $m$ in $T_m(u)$ means that we act on the $m$-th copy of the tensor product. Both the left-hand side and the hand-side belong to $\mathrm{Y}(N)[[u^{-1},v^{-1}]]\otimes\textrm{End}(W^{\otimes2})$ (localized by a certain multiplicative system), i.e. they can be regarded as operators on $W^{\otimes2}$ with the coefficients in $\textrm{Y}(N)$. This relation is usually referred to as \textit{ternary relation.}

\subsection{Twisted Yangian $\textrm{Y}^{\mp}(N)$.}

In this subsection we recall some standard facts about twisted Yangians $\textrm{Y}^{\mp}(N)$ following \cite{MNO}.

For the symplectic case we set $N=2n$ and consider a non-degenerate skew-symmetric bilinear form\\ $\langle\cdot,\cdot\rangle_{-}$ on $W=\mathbb{C}^N$ defined on basis vectors by the formulas:
\begin{equation}\label{skew_form}
\langle e_i, e_j\rangle_{-}=sgn(i)\delta_{i,-j},
\end{equation}
where $i,j=-n,...,-1,1,...,n$. Denote by $t$ the transposition associated with the form ($\ref{skew_form}$), i.e. 
\begin{equation}\label{transposition}
(E_{ij})^t=\theta_{ij}E_{-j,-i},
\end{equation}
where $\theta_{ij}:=sgn(i)sgn(j)$.

In the orthogonal case we consider a non-degenerate symmetric bilinear form and the corresponding transposition on $\textrm{End}(W)$ is defined by the same formulas ($\ref{transposition}$) but with $\theta_{ij}:=1$ for $i,j=-n,...,-1,0,1,...,n$.

For both cases introduce the $S$-matrix 
\begin{equation}\label{S-matrix}
S(u):=T(u)T^t(-u)\in\textrm{Y}(N)\otimes\textrm{End}(W),
\end{equation}
where $T^t(u)=\sum^n_{i,j=-n} \theta_{ij}t_{-j,-i}(u)\otimes E_{ij}$ is the image of usual T-matrix (but with indices $1,...,N$ changed to $-n,...,n$) under transposition $t$. The $(i,j)$-th entry of $S(u)$ is a power series in $u^{-1}$:
\begin{equation}\label{S-series}
s_{ij}(u)=\delta_{ij}+s^{(1)}_{ij}u^{-1}+s^{(2)}_{ij}u^{-2}+...
\end{equation}

The series $s_{ij}(u)$ can be expressed in terms of $t_{ij}(u)$ as
\begin{equation}\label{S_series}
s_{ij}(u)=\sum^n_{a=-n}\theta_{aj}t_{ia}(u)t_{-j,-a}(-u),
\end{equation}
and its coefficients have the following presentation:
\begin{equation}\label{texisted_generators}
s^{(M)}_{ij}=\sum^n_{a=-n}\sum^M_{r=0}\theta_{aj}t^{(M-r)}_{ia}(-1)^rt^{(r)}_{-j,-a}.
\end{equation}

\begin{defn} \textit{Twisted Yangian} $\textrm{Y}^{\mp}(N)$ is the subalgebra of Yangian $\textrm{Y}(N)$, generated by all $s^{(M)}_{ij}$ (where ``-'' corresponds to the construction via transposition associated with the skew-symmetric bilinear form and ``+'' corresponds to the symmetric form).
\end{defn}

\begin{rem} From now on every time we have symbols ``$\pm$'' or ``$\mp$'' inside the formulas the upper sign corresponds to $\textrm{Y}^-(N)$ while the lower sign corresponds to $\textrm{Y}^+(N)$.
\end{rem}

\begin{prop} 
(\cite{MNO}, Proposition 3.7 and Theorem 3.6) The generators of twisted Yangian $\textrm{Y}^{\mp}(N)$ satisfy following commutation relations:
\begin{eqnarray}\label{scruch3}
[s_{ij}(u),s_{kl}(v)]=\frac{1}{u-v}(s_{kj}(u)s_{il}(v)-s_{kj}(v)s_{il}(u))-\notag\\
-\frac{1}{u+v}\cdot(\theta_{k,-j}s_{i,-k}(u)s_{-j,l}(v)-\theta_{i,-l}s_{k,-i}(v)s_{-l,j}(u))+\\
+\frac{1}{u^2-v^2}\theta_{i,-j}(s_{k,-i}(u)s_{-j,l}(v)-s_{k,-i}(v)s_{-j,l}(u)).\notag\\
\notag
\end{eqnarray}
\begin{center}
and
\end{center}
\begin{equation}\label{scruch4}
\theta_{ij}s_{-j,-i}(-u)=s_{ij}(u)\mp\frac{s_{ij}(u)-s_{ij}(-u)}{2u}
\end{equation}
for $|i|,|j|,|k|,|l|\leq n$.
\end{prop}


\subsection{Maps of the twisted Yangian $\textrm{Y}^{\mp}(N)$.}
Throughout the whole paper when we work with $\textrm{Y}^-(N)$ we assume $N=2n$, $\mathfrak{g}_n=\mathfrak{sp}_{2n}=\mathfrak{sp}_N$, and for $\textrm{Y}^+(N)$ we assume $N=2n+1$, $\mathfrak{g}_n=\mathfrak{o}_{2n+1}=\mathfrak{o}_N$.

There exist maps between $\textrm{Y}^{\mp}(N)$ and $U(\mathfrak{g}_n)$ which allow us to study representations of one algebra via representations of the other.\\
Recall that $F_{ij}:=E_{ij}-\theta_{ij}E_{-j,-i}$  ($-n\leq i,j\leq n$) span the subalgebra $\mathfrak{g}_n$ in $\mathfrak{gl}(N)$, then the following proposition holds (\cite{MNO}, Propositions 3.11 and 3.12).

\begin{prop}\label{pr:EvaluationMap}
(i) The map $$\xi:s_{ij}(u)\mapsto\delta_{ij}+(u\mp\frac{1}{2})^{-1}F_{ij}$$ defines an algebra homomorphism
\begin{equation}\label{xi_map_scruch}
\xi:\textrm{Y}^{\mp}(N)\rightarrow U(\mathfrak{g}_n).
\end{equation}

(ii) The map
\begin{equation}
\nu:F_{ij}\mapsto s^{(1)}_{ij}
\end{equation}
is an embedding of $U(\mathfrak{g}_n)$ into $\textrm{Y}^{\mp}(N)$.
\end{prop}

\subsection{Commutative subalgebras of $\textrm{Y}^{\mp}(N)$.}
To be able to describe the construction of Bethe subalgebras in $\textrm{Y}^{\mp}(N)$ (following \cite{NO}) we need to introduce few more notations. First of all, we work with tensors from $\textrm{Y}^{\mp}(N)[[u_1^{-1},...,u_m^{-1}]]\otimes\textrm{End}(W^{\otimes N})$. For the operators $$R_{ij}(u_i-u_j):=1-\frac{P_{ij}}{u_i-u_j},$$ where $P_{ij}$ is just permuting $i$-th and $j$-th terms of the tensors, to be correctly defined on the above space we need to localize this space by a multiplicative system $\{(u_i^{-1}-u_j^{-1}), (u_i^{-1}+u_j^{-1})|-n\leq i,j\geq n;\;i\neq j\}$. Such localization will allow us to work with specializations of the form $u_k=u_l+a$ when $a\neq0$ for some $-n\leq k,l\geq n$ as well.\\

Set $S_i:=S_i(u-i+1)$ to be an operator from the localized $\textrm{Y}^{\mp}(N)[[u_1^{-1},...,u_m^{-1}]]\otimes\textrm{End}(W^{\otimes N})$ acting on the $i$-th copy of $W$ as the S-matrix $S(u-i+1)$ (here we abuse the fact that localizations are possible) and as 1 on all other terms:
\begin{equation}
S_i=\sum_{i,j=-n}^ns_{ij}(u-i+1)\otimes1^{\otimes(i-1)}\otimes E_{ij}\otimes1^{\otimes(n-i)}.
\end{equation}

By $R^t(u)$ we denote the following operator from $\textrm{Y}^{\mp}(N)[[u^{-1}]]\otimes\textrm{End}(W^{\otimes2})$:
\begin{equation}
R^t(u):=1-\frac{\sum_{i,j=-n}^nE_{ij}^t\otimes E_{ji}}{u}=1-\frac{\sum_{i,j=-n}^nE_{ij}\otimes E_{ji}^t}{u},
\end{equation}
where $t:\textrm{End}(W^{\otimes N})\rightarrow\textrm{End}(W^{\otimes N})$ is the transposition defined earlier in Subsection 1.2.

Then $R'_{ij}=R^t_{ij}(-2u+i+j-2)$ is another element of the localization of $\textrm{Y}^{\mp}(N)[[u_1^{-1},...,u_m^{-1}]]\otimes\textrm{End}(W^{\otimes N})$ acting as $R^t(-2u+i+j-2)$ on the $i$-th and the $j$-th copies of tensor product.

Finally, the following elements of $\textrm{Y}^{\mp}[[u^{-1}]]\otimes\textrm{End}(W^{\otimes N})$ can be introduced:
\begin{equation}\label{S_uk}
S(u,k)=S_1(R'_{12}\cdot...\cdot R'_{1k})S_2(R'_{23}\cdot...\cdot R'_{2k})\cdot...\cdot S_k
\end{equation}
\begin{center}
and
\end{center}
\begin{equation}\label{C_uk}
C(u,k)=C_{k+1}\tilde{R}^t_{k+1,k+2}\cdot...\cdot\tilde{R}^t_{k+2,N}C_{k+2}\cdot...\cdot C_{N-1}\tilde{R}^t_{N-1,N}C_N,
\end{equation}
where $\tilde{R}^t_{ij}=R^t_{ij}(-2u-N+i+j+2)$ and $C\in\textrm{End}(W)$ satisfies $C^t=-C$.

Consider the following series with the coefficients in $\textrm{Y}^{\mp}(N)$:
\begin{equation}
\sigma_k(u,C)=\textrm{tr}(A_NS(u,k)(\prod^{\rightarrow}_{i=1,...,k}\;\prod^{\rightarrow}_{j=k+1,...,N}R^t_{ij})C(u,k)),
\end{equation}
where $A_N$ denotes the image of normalised  antisymmetrizer $a_N =\frac{1}{N!}\sum_{p\in\mathfrak{S}_N}sgn(p)p$ under the natural map from the symmetric group $\mathfrak{S}_N$ to $\textrm{End}(W^{\otimes N})$.

The main results about these series for $C\in\textrm{End}(W^{\otimes N})$ satisfying $C^t=-C$ (\cite{NO}, Proposition 3.3, Theorem 3.4 and Theorem 3.5) are gathered in the Theorem below.

\begin{predl}\label{th:BetheSubalgebra}
(i) The coefficients of $\sigma_N(u,C)$  generate the center of $Y^{\mp}(N)$.

(ii) The coefficients of $\sigma_1(u,C),...,\sigma_N(u,C)$ generate a commutative subalgebra of $Y^{\mp}(N)$ called \emph{Bethe subalgebra}.

(iii) If $C$ has pairwise distinct eigenvalues, then the coefficients at $u^{-2}, u^{-4},...$ of the series $\sigma_N(u,C)$, $\sigma_{N-2}(u,C)$, ...  and the coefficients at $u^{-1}, u^{-3},...$ of the series $\sigma_{N-1}(u,C)$, $\sigma_{N-3}(u,C)$,... are free generators of this commutative subalgebra. This commutative subalgebra is maximal.
\end{predl}
In this paper we restrict to the Bethe subalgebras corresponding to regular Cartan (i.e. diagonal) $C$.

\subsection{Commutative subalgebras in $\textrm{Y}^{\mp}(2)$.}
For $\mathfrak{sp}_2$ the regular Cartan element $F_{11}=E_{11}-E_{-1,-1}$ is unique up to a constant factor. Therefore, the maximal commutative subalgebra (Bethe subalgebra) provided by the above construction is unique as well. We denote this subalgebra by $\mathcal{B}^-$. By Theorem~\ref{th:BetheSubalgebra}(i), the coefficients of $\sigma_2(u,F_{11})$ generate the center of $\textrm{Y}^-(2)$. Other generators of $\mathcal{B}^-$ are the coefficients of $\sigma_1(u, F_{11})$.

By definition,
\begin{equation}\label{sigma2}
\sigma_1(u, C)=\textrm{tr} \left[A_2\cdot S_1\cdot (1-\frac{\sum_{i,j}E^t_{ij}\otimes E_{ji}}{3-2u})\cdot C_2\right],
\end{equation}
where $A_2=\frac{1}{2!}(1-\sum_{i,j}E_{ij}\otimes E_{ji})$. 

The image of $A_2$ is one-dimensional -  $\mathbb{C}\cdot(e_{-1}\otimes e_1-e_1\otimes e_{-1})$, hence the basis vectors $e_{-1}\otimes e_{-1}$ and $e_1\otimes e_1$ have zero contribution to the trace in ($\ref{sigma2}$). For the remaining basis vectors before evaluating the trace we get the following images:
$$e_{-1}\otimes e_1\mapsto\frac{1}{2}\left(s_{-1,-1}(u)-\frac{1}{3-2u}(s_{-1,-1}(u)+s_{11}(u))\right)(e_{-1}\otimes e_1-e_1\otimes e_{-1}),$$
$$e_1\otimes e_{-1}\mapsto\frac{1}{2}\left(-s_{11}(u)+\frac{1}{3-2u}(s_{11}(u)+s_{-1,-1}(u))\right)(e_1\otimes e_{-1}-e_{-1}\otimes e_1).$$
Hence we have $$\sigma_1(u, F_{11})=\frac{1}{2}(s_{-1,-1}(u)-s_{11}(u)),$$ so the elements $s^{(2m+1)}_{11}-s^{(2m+1)}_{-1,-1}$ ($m\in\mathbb{Z}_{\geq0}$) together with the center of $\textrm{Y}^{-}(2)$ generate a maximal commutative subalgebra in $\textrm{Y}^{-}(2)$ by Theorem~\ref{th:BetheSubalgebra} (iii). From the \textit{symmetry relation} ($\ref{scruch4}$) it follows that $s_{11}^{(2m+1)}=-s_{-1,-1}^{(2m+1)}$. Thus we can state that $\mathcal{B}^-$ is generated by the center of $\textrm{Y}^-(2)$ and the elements $s_{11}^{(2m+1)}$ ($m\in\mathbb{Z}_{\geq0}$).

A similar situation occurs for $\mathfrak{o}_2$. The element $F_{11}$ defines the Bethe subalgebra $\mathcal{B}^+$ of $Y^+(2)$. The non-trivial generators are the coefficients of $\sigma_1(u,C)$ but in this case $E^t_{ij}=E_{-j,-i}$ leading to:
$$e_{-1}\otimes e_1\mapsto\frac{1}{2}\left(s_{-1,-1}(u)+\frac{1}{3-2u}\left(s_{11}(u)-s_{-1,-1}(u)\right)\right)e_{-1}\otimes e_1\;+\;\ldots,$$
$$e_1\otimes e_{-1}\mapsto\frac{1}{2}\left(-s_{11}(u)+\frac{1}{3-2u}\left(s_{11}(u)-s_{-1,-1}(u)\right)\right)e_1\otimes e_{-1}\;+\;\ldots$$
After taking the trace we obtain $$\sigma_1(u,F_{11})=\frac{2u-1}{6-4u}(s_{11}(u)-s_{-1,-1}(u)).$$
Again from the symmetry relation ($\ref{scruch4}$) for $\textrm{Y}^+(N)$ we know that $s_{11}^{(2m+1)}=-s_{-1,-1}^{(2m+1)}$ ($m\in\mathbb{Z}_{\geq0}$) and $s_{11}^{(2m)}-s_{-1,-1}^{(2m)}=s_{-1,-1}^{(2m-1)}$ ($m\in\mathbb{Z}_{>0}$). This observation allows us to state that the subalgebra generated by the coefficients of $\sigma_1(u, F_{11})$ at $u^{-1}$, $u^{-3}$, ... coincides with the subalgebra generated by $s_{11}^{(2m+1)}$ ($m\in\mathbb{Z}_{\geq0}$). Together with the central elements of $\textrm{Y}^+(2)$ they generate Bethe subalgebra $\mathcal{B}^+\subset\textrm{Y}^+(2)$.

The latter results can be united in the following proposition.

\begin{prop}
Bethe subalgebra $\mathcal{B}^{\mp}\subset\textrm{Y}^{\mp}(2)$ corresponding to $C=F_{11}$ is generated by all $s_{11}^{(2m+1)}$ with $m\in\mathbb{Z}_{\geq0}$ and the center $ZY^{\mp}(2)$ of the twisted Yangian $Y^{\mp}(2)$.
\end{prop}

\subsection{Representation theory of Yangian $\textrm{Y}(2)$ and twisted Yangians $\textrm{Y}^{\mp}(2)$.}
The Yangian $\textrm{Y}(2)$ is a Hopf algebra with the coproduct $\Delta$ on $\mathrm{Y}(2)$ determined by
\begin{equation}\label{coproduct}
\Delta\left(t_{ij}\left(u\right)\right)=t_{i1}(u)\otimes t_{1j}(u)+t_{i,-1}(u)\otimes t_{-1,j}(u).
\end{equation}
This defines a $\textrm{Y}(2)$-module structure on any tensor products of $\textrm{Y}(2)$-modules. Moreover, the twisted Yangians $\textrm{Y}^{\mp}(2)$ are left coideal subalgebras in $\textrm{Y}(2)$ with respect to $\Delta$, so any tensor product of a $\textrm{Y}(2)$-module by a $\textrm{Y}^{\mp}(2)$ is still a $\textrm{Y}^{\mp}(2)$-module. We will construct all necessary $\textrm{Y}^{\mp}(2)$-modules by tensoring very simple ones using the above Hopf coideal structure.  

For a pair of complex numbers $(\alpha, \beta)$ with $\alpha-\beta\in\mathbb{Z}_+$ we consider the irreducible representation $L(\alpha, \beta)$ of $\mathfrak{gl}_2$ with highest weight $(\alpha,\beta)$. One can define the action of the Yangian $\mathrm{Y}(2)$ on $L(\alpha,\beta)$ via the evaluation homomorphism:
\begin{equation}\label{yangian_action}
\mathrm{Y}(2)\rightarrow U\left(\mathfrak{gl}_2\right)
\end{equation}
\begin{equation}
t_{ij}(u)\mapsto\delta_{ij}+E_{ij}u^{-1},\;\;i,j=-1,1
\end{equation}

The coproduct $\Delta$ on $\mathrm{Y}(2)$ allows us to construct algebra homomorphism $\mathrm{Y}(2)\rightarrow U\left(\mathfrak{gl}_2\right)^{\otimes k}$ for $k\in\mathbb{Z}_+$. This homomorphism endows the $\mathfrak{gl}_2$-module 
\begin{equation}\label{irr_modules}
L=L(\alpha_1,\beta_1)\otimes...\otimes L(\alpha_k,\beta_k)
\end{equation}
with a structure of a $\mathrm{Y}(2)$-module. We can restrict $L$ to the twisted Yangian $\mathrm{Y}^-(2)$ using the expression (\ref{S_series}) of the generators $s_{ij}^{(M)}$ ($i,j=-1,1$; $M\in\mathbb{Z}_{\geq0}$) in terms of the generators of $\mathrm{Y}(2)$:
\begin{equation}\label{s_t_relation_2}
s_{ij}(u)=\theta_{1j}t_{i1}(u)t_{-j,-1}(-u)+\theta_{-1,j}t_{i,-1}(u)t_{-j,1}(-u)
\end{equation}

Next, for any $\delta\in\mathbb{C}$ we have a one-dimensional representation $W\left(\delta\right)$ of $\mathrm{Y}^+(2)$ spanned by vector $w$ with
\begin{equation}\label{1_dim_repr_+}
s_{11}(u)w=\frac{u+\delta}{u+1/2}w,\;\;\;s_{-1,-1}(u)w=\frac{u-\delta+1}{u+1/2}w,
\end{equation}
and $s_{1,-1}(u)w=s_{-1,1}(u)w=0$. From the Hopf coideal structure on $\mathrm{Y}^+(2)$ we have a structure of $\mathrm{Y}^+(2)$-module on $L\otimes W\left(\delta\right)$.

We will compute the images of non-trivial generators $s_{11}^{(2m+1)}$ ($m\in\mathbb{Z}_{\geq0}$) of $\mathcal{B}^-$ under homomorphism
\begin{equation}
\mathrm{Y}^-(2)\rightarrow U(\mathfrak{gl}_2)^{\otimes k}
\end{equation}
and of $\mathcal{B}^+$ under homomorphism
\begin{equation}
\mathrm{Y}^+(2)\rightarrow U(\mathfrak{gl}_2)^{\otimes k}\otimes \mathbb{C}[\delta]
\end{equation}
in the last section of this paper.

\subsection{Homomorphism to the centralizer algebra.}




One of the main features of twisted Yangians is the existence of evaluation homomorphisms (contrary to the usual Yangians of classical Lie algebras except type A). To define such homomorphisms, consider the following matrix with the coefficients from $U(\mathfrak{g}_n)$:
\begin{equation}\label{calF-matrix}
\mathcal{F}:=(F_{ij})_{i,j=-n}^n=\sum_{i,j=-n}^nF_{ij}\otimes E_{ij}\in\mathfrak{g}_n\otimes\textrm{End}(W)
\end{equation} 

Consider series with coefficients from the algebra of polynomial functions in the coordinates $l_i=\lambda_i+\rho_i$  (given by the components of a weight shifted by $\rho$) on the dual of the Cartan subalgebra $(\mathfrak{h}_n)^*$, which are invariant under the ``shifted'' action of the Weyl group:
$$\chi_n(u):=\prod^n_{i=1}\frac{(u+1/2)^2-l^2_i}{(u+1/2)^2-\rho_i^2}.$$ The coefficients of these series can be regarded as central elements of the universal enveloping algebra via Harish-Chandra homomorphism. In fact the only property of $\chi_n$ we need in this paper is that it is a series of the form 
\begin{equation}\label{def:chi} \chi_n(u)=1+\sum\limits_{r=1}^\infty \chi_{n,r}u^{-r},
\end{equation} where $\chi_{n,r}$ are central elements of $U(\mathfrak{g}_n)$ of the PBW degree $r$.

We have the following (see \cite{MO}, Proposition 4.14 and Proposition 4.15).

\begin{predl}\label{th:centralizer}
(i) The map $$\varphi_n:S(u)\mapsto \chi_n(u)(1-\frac{\mathcal{F}}{u+\frac{N\pm1}{2}})^{-1}$$ is a surjective homomorphism of algebras $\textrm{Y}^{\mp}(N)\rightarrow U(\mathfrak{g}_{n})$.

(ii) For $m<n$ the image of $Y^{\mp}(M)$ under $\varphi_n$ is contained in the centralizer subalgebra $U(\mathfrak{g}_{n})^{\mathfrak{g}_{n-m}}$.
\end{predl}

\begin{rem} \emph{Here $\mathrm{Y}^{\mp}(M)$ is naturally embedded in $\mathrm{Y}^{\mp}(N)$ as a subalgebra generated by all $s_{ij}^{(L)}$ with $|i|,|j|=n-m+1,...,n$ and $L\in\mathbb{Z}_{\geq0}$ when $M$ is even. In case of odd $M$ subalgebra $\mathrm{Y}^+(M)$ is generated by $s_{ij}^{(L)}$ with $|i|,|j|=0,n-m+1,...,n$.}
\end{rem}

In particular, we have a natural homomorphism $\varphi_n:\mathrm{Y}^\mp(2)\to U(\mathfrak{g}_{n})^{\mathfrak{g}_{n-1}}$. The centralizer subalgebra $U(\mathfrak{g}_{n})^{\mathfrak{g}_{n-1}}$ acts naturally on any multiplicity space of the restriction of a $\mathfrak{g}_n$-module to $\mathfrak{g}_{n-1}$, so all such multiplicity spaces are naturally $\mathrm{Y}^\mp(2)$-modules. From Theorem 3.15 (ii) of $\cite{M}$ we have the following statement.

\begin{predl}\label{th:GTsp}
Let $\lambda=(\lambda_1,...,\lambda_n)$ and $\mu=(\mu_1,...,\mu_n)$ be highest weights of finite-dimensional irreducible representations of $\mathfrak{sp}_{2n}$ and $\mathfrak{sp}_{2n-2}$ and $V_\lambda^{\mu}$ denote the corresponding multiplicity space. Then the action of $\mathrm{Y}^-(2)$ on $V_\lambda^{\mu}$ defined as a composition of homomorphism $\varphi_n$ to $U(\mathfrak{sp}_{2n})^{\mathfrak{sp}_{2n-2}}$ and a natural projection is irreducible and isomorphic to
\begin{equation}
L(\alpha_1,\beta_1)\otimes...\otimes L(\alpha_n,\beta_n),
\end{equation}
where $\alpha_1=-1/2$ and 
\begin{equation}
\alpha_{i}=\min\{\lambda_{i-1},\mu_{i-1}\}-i+1/2,\;i=2,...,n,
\end{equation}
\begin{equation}
\beta_i=\max\{\lambda_i,\mu_i\}-i+1/2,\;i=1,...,n.
\end{equation}
\end{predl}

Similarly, Theorem 3.14 (ii) of $\cite{M}$ implies the following:

\begin{predl}\label{th:GTo}
Let $\lambda=(\lambda_1,...,\lambda_n)$ and $\mu=(\mu_1,...,\mu_n)$ be highest weights of finite-dimensional irreducible representations of $\mathfrak{o}_{2n+1}$ and $\mathfrak{o}_{2n-1}$ and $V_\lambda^{\mu}$ denote the corresponding multiplicity space. Then the action of $\mathrm{Y}^+(2)$ on $V_\lambda^{\mu}$ defined as a composition of homomorphism $\varphi_n$ to $U(\mathfrak{o}_{2n+1})^{\mathfrak{o}_{2n-1}}$ and a natural projection is isomorphic to the direct sum of two irreducible submodules, $V_\lambda^{\mu}\simeq U\oplus U'$, where
\begin{equation}\label{first_direct_summand_first_case}
U=L\left(0,\beta_1\right)\otimes L\left(\alpha_2,\beta_2\right)\otimes...\otimes L\left(\alpha_n,\beta_n\right)\otimes W\left(1/2\right),
\end{equation}
\begin{equation}\label{second_direct_summand_first_case}
U'=L\left(-1,\beta_1\right)\otimes L\left(\alpha_2,\beta_2\right)\otimes...\otimes L\left(\alpha_n,\beta_n\right)\otimes W\left(1/2\right),
\end{equation}
if the $\lambda_i$ are integers (it is supposed that $U'=\{0\}$ if $\beta_1=0$); or
\begin{equation}\label{first_direct_summand_second_case}
U=L\left(-1/2,\beta_1\right)\otimes L\left(\alpha_2,\beta_2\right)\otimes...\otimes L\left(\alpha_n,\beta_n\right)\otimes W\left(0\right),
\end{equation}
\begin{equation}\label{second_direct_summand_second_case}
U'=L\left(-1/2,\beta_1\right)\otimes L\left(\alpha_2,\beta_2\right)\otimes...\otimes L\left(\alpha_n,\beta_n\right)\otimes W\left(1\right),
\end{equation}
if the $\lambda_i$ are half-integers, and the following notation is used
\begin{equation}
\alpha_{i}=\min\{\lambda_{i-1},\mu_{i-1}\}-i+1,\;i=2,...,n,
\end{equation}
\begin{equation}
\beta_i=\max\{\lambda_i,\mu_i\}-i+1,\;i=1,...,n.
\end{equation}
\end{predl}

\section{Shift of argument subalgebras and their quantization.}
\subsection{Construction of Poisson-commutative subalgebras.}

A semisimple complex Lie algebra $\mathfrak{g}$ can be identified with its dual space $\mathfrak{g}^*$ via the Killing form, therefore we can think of Poisson algebra $S(\mathfrak{g})$ as the space of functions on $\mathfrak{g}$. By the classical result of Chevalley, the Poisson center $S(\mathfrak{g})^{\mathfrak{g}}$ is a free polynomial algebra in $n=\text{rk}\ \mathfrak{g}$ generators $\Phi_1,\ldots,\Phi_n$. For $\mathfrak{g}=\mathfrak{sp}_{2n}$ and $\mathfrak{g}=\mathfrak{o}_{2n+1}$ we can write these generators explicitly as follows. Consider the following matrix with the coefficients from $S(\mathfrak{g})$:
\begin{equation}\label{F-matrix}
F:=(F_{ij})_{i,j=-n}^n=\sum_{i,j=-n}^nF_{ij}\otimes E_{ij}\in\mathfrak{g}\otimes\textrm{End}(W)
\end{equation} 
Then the following elements are free generators of $S(\mathfrak{g})^{\mathfrak{g}}$:

\begin{equation}\label{center_generators}
\Phi_r=\text{Tr}\ F^{2r} \;(r=1,2,\ldots,n)
\end{equation}


\begin{rem} \emph{When appropriate we regard the matrix $F$ and its powers as elements of $S(\mathfrak{g})\otimes\textrm{End}(W)$ or as elements of $U(\mathfrak{g})\otimes\textrm{End}(W)$. In the latter case we denote it by $\mathcal{F}$.}
\end{rem}

For any element $\mu\in\mathfrak{g}$ we define a Poisson-commutative subalgebra $A_{\mu}$ called \textit{shift of argument} or \textit{Mischenko-Fomenko algebra}. This algebra is generated by all possible derivatives of $\Phi_r$ along $\mu$, i.e. by all elements of the form $\partial_{\mu}^k\Phi_r$.

Consider regular power series in the sense of Vinberg's work \cite{V1} $\mu(\varepsilon)=\mu_0+\mu_1\varepsilon+...+\mu_l\varepsilon^l$. They are such series from $\mathfrak{g}\otimes\mathbb{C}[[\varepsilon]]$ that for sufficiently small nonzero values of $\varepsilon$ elements $\mu(\varepsilon)$ are regular and semisimple. 

The dimension of the $k$-th graded component $(A_{\mu(\varepsilon)})_k=A_{\mu(\varepsilon)}\cap S^k(\mathfrak{g})$ does not depend on $\varepsilon$ for sufficiently small nonzero values of $\varepsilon$. Since the coefficients of the derivatives of $\mathfrak{g}$-invariants are convergent power series in $\varepsilon$, it follows that Plücker coordinates of subspace $(A_{\mu(\varepsilon)})_k\subset S^k(\mathfrak{g})$ are convergent power series in $\varepsilon$. Hence there is a well defined limit $A_k:=\lim_{\varepsilon\to0}(A_{\mu(\varepsilon)})_k$ in the corresponding Grassmann variety. That allows us to define the commutative subalgebra $\lim_{\varepsilon\to0}A_{\mu(\varepsilon)}:=\oplus A_k$. It is called \textit{limit shift of argument subalgebra}.

The following theorem provides us with maximal Poisson-commutative subalgebras constructed from regular semisimple $\mu\in\mathfrak{g}$ and regular series $\mu(\varepsilon)=\mu_0+\mu_1\varepsilon+...+\mu_l\varepsilon^l$ such that all $\mu_i$ belong to some fixed Cartan subalgebra $\mathfrak{h}$ and $\mathfrak{z}(\mu_0)\cap\mathfrak{z}(\mu_1)\cap...\cap\mathfrak{z}(\mu_l)=\mathfrak{h}$ (\cite{T}, Theorem 1 and Theorem 2).

\begin{predl}
(i) For regular semisimple element $\mu\in\mathfrak{g}$ all partial derivatives $\partial^k_{\mu}\Phi_r$ with $k=0,1,\ldots,\deg \Phi_r$ freely generate a maximal commutative subalgebra $A_{\mu}$ of transcendence degree equal to $\frac{1}{2}(\dim\mathfrak{g}+\text{rk}\ \mathfrak{g})$ in the Poisson algebra $S(\mathfrak{g})$.

(ii) Limit shift of argument subalgebras $\lim\limits_{\varepsilon\to0}A_{\mu(\varepsilon)}$ are maximal commutative subalgebras in $S(\mathfrak{g})$ of the same transcendence degree $\frac{1}{2}(\dim\mathfrak{g}+\text{rk}\ \mathfrak{g})$.
\end{predl}

\subsection{Explicit description of limit shift of argument subalgebras.}

In \cite{Sh} Shuvalov gives an explicit description of the limit shift of argument subalgebras in the following terms. Consider a regular power series $\mu(\varepsilon)=\mu_0+\mu_1\varepsilon+...+\mu_l\varepsilon^l$ and set $$\mathfrak{z}_i=\mathfrak{z}_{\mathfrak{g}}(\mu_0)\cap...\cap\mathfrak{z}_{\mathfrak{g}}(\mu_i)\;\;(i=0,...l),\;\mathfrak{z}_{-1}=\mathfrak{g}.$$

Clearly, we have $\mathfrak{z}_{l}=\mathfrak{h}$. We define subalgebras $A_k$ ($k=0,...,l+1$) of Poisson algebra $S(\mathfrak{g})$ inductively:
\begin{itemize}
\item $A_0=A_{\mu_0}$;
\item $A_k$ is the subalgebra generated by $A_{k-1}$ and the derivatives of the invariants of $\mathfrak{z}_{k-1}$ in $S(\mathfrak{z}_{k-1})$ along $\mu_k$;
\item $A_{l+1}$ is the subalgebra generated by $A_l$ and $\mathfrak{z}_l=\mathfrak{h}$.
\end{itemize}

The following result holds for regular power series $\mu(\varepsilon)$ with all $\mu_i$'s lying in some Cartan subalgebra $\mathfrak{h}\subset\mathfrak{g}$ and satisfying $\mathfrak{z}(a_0)\cap\mathfrak{z}(\mu_1)\cap...\cap\mathfrak{z}(\mu_l)=\mathfrak{h}$ (\cite{Sh}, Theorem 1).

\begin{predl}
(i) For reductive complex Lie algebra $\mathfrak{g}$ and regular series $\mu(\varepsilon)=\mu_0+\mu_1\varepsilon+...+\mu_l\varepsilon^l$ the limit shift of argument subalgebra coincide with $A_{l+1}$: $$\lim\limits_{\varepsilon\to0}A_{\mu(\varepsilon)}=A_{l+1}.$$
(ii) Limit shift of argument subalgebra $\lim\limits_{\varepsilon\to0}A_{\mu(\varepsilon)}$ is freely generated by some of the derivatives of the invariants of $\mathfrak{z}_{k-1}$ along $\mu_k$ for $k=0,\ldots,l$ and by $\mathfrak{z}_l=\mathfrak{h}$.
\end{predl}

\subsection{Quantization of shift of argument subalgebras.}
The problem of lifting shift of argument commutative subalgebras to universal enveloping algebra was solved in the case of regular semisimple $\mu\in\mathfrak{g}^*$ in \cite{R}. Here is the key theorem which provides us with quantization of $A_{\mu}$.

\begin{predl}\label{th:quantization}\cite{R} (i) For regular $\mu\in\mathfrak{g}^*$ there exists commutative subalgebra $\mathcal{A}_{\mu}\subset U(\mathfrak{g})$ of the universal enveloping algebra such that $gr\mathcal{A}_{\mu}=A_{\mu}$.\\
(ii) Moreover, there exists commutative subalgebra $\widehat{A}\subset(U(\mathfrak{g})\otimes S(\mathfrak{g}))^{\mathfrak{g}}$, such that for evaluation map at $\mu\in\mathfrak{g}^*$:
\begin{equation}
ev_{\mu}:U(\mathfrak{g})\otimes S(\mathfrak{g})=U(\mathfrak{g})\otimes\mathbb{C}[\mathfrak{g}^*]\longrightarrow U(\mathfrak{g}),
\end{equation}
we have $ev_{\mu}(\widehat{A})=\mathcal{A}_{\mu}$.
\end{predl}

A result similar to Theorem~\ref{th:quantization} (ii) holds for limit shift of argument subalgebras as well. It has been proved in \cite{KR} (Theorem 10.7 and Theorem 10.8). Namely, the lift $\lim\limits_{\varepsilon\rightarrow0}\mathcal{A}_{\mu\left(\varepsilon\right)}$ of the limit Mischenko-Fomenko algebra corresponding to regular series $\mu(\varepsilon)=\mu_0+\mu_1\varepsilon+...+\mu_l\varepsilon^{l}$ can be constructed inductively:
\begin{itemize}
\item $\mathcal{A}_0=ev_{\mu_0}(\widehat{A}_0)\subset U\left(\mathfrak{g}\right)^{\mathfrak{z}_0}$ ($\widehat{A}_0$ is provided by Theorem 5(ii));
\item By Theorem 5, we have the universal quantum shift of argument subalgebra $\widehat{A}_k\subset\left[U\left(\mathfrak{z}_k\right)\otimes S\left(\mathfrak{z}_k\right)\right]^{\mathfrak{z}_k}$, then $\mathcal{A}_{k+1}=ev_{\mu_{k+1}}(\widehat{A}_k)\cdot\mathcal{A}_{k}$.
\end{itemize}
\begin{predl}\label{th:quantShuvalov}\cite{KR} Quantization of the limit shift of argument subalgebra coincides with $\mathcal{A}_l$: $$\lim_{\varepsilon\rightarrow0}\mathcal{A}_{\mu\left(\varepsilon\right)}=\mathcal{A}_l.$$
\end{predl}

\section{Proof of the theorem A.}
\subsection{Isotypic components.}

Consider the action of $\mathfrak{g}$-invariants $S(\mathfrak{g)}^{\mathfrak{g}}=\mathbb{C}[\Phi_1,\ldots,\Phi_n]$ on $S(\mathfrak{g})$. We have the following well-known result of Kostant (\cite{K}, Theorem 0.12).

\begin{predl}
[Kostant] $S(\mathfrak{g})$ is a free $S(\mathfrak{g)}^{\mathfrak{g}}$-module, $S(\mathfrak{g})=S(\mathfrak{g})^{\mathfrak{g}}\otimes H$, where the space of generators $H=\bigoplus m_{\lambda}V_{\lambda}$ is the sum of all irreducible finite-dimensional representations $V_{\lambda}$ of Lie algebra $\mathfrak{g}$ with highest weight $\lambda$ taken with the multiplicity $m_{\lambda}=\dim V_{\lambda}(0)$, where $V_{\lambda}(0)$ is the $0$-weight subspace of $V_\lambda$.
\end{predl}

In particular, the isotypic component of $V_{\lambda}$ in $S(\mathfrak{g})$ is a direct sum of $\dim V_\lambda(0)$ copies of $S(\mathfrak{g)}^{\mathfrak{g}}\otimes V_\lambda$. The space of generators of each copy is homogeneous of some degree. We will be interested in such isotypic component for $V_\lambda$ being the adjoint representation. The multiplicity $\dim V_\lambda(0)$ is in this case just the rank of $\mathfrak{g}$, and the degrees are the exponents of $\mathfrak{g}$, i.e. $m_r:=\deg\Phi_r-1$ for $r=1,\ldots,n$. The space of degree $m_r$ generators of this isotypic component can be given in the following way. Consider the following homomorphism of $\mathfrak{g}$-modules:
\begin{equation}
\tau: \mathfrak{g}\otimes S(\mathfrak{g})\longrightarrow S(\mathfrak{g}),
\end{equation}
\begin{center}
$\mu\otimes P\mapsto \partial_{\mu}P.$
\end{center}

Under this homomorphism $\mu\otimes\Phi_r\mapsto\partial_{\mu}\Phi_r$ ($r=1,\ldots,\textrm{rk}\ \mathfrak{g}$) and since $\Phi_r$ is a central element we have $[g,\mu\otimes\Phi_r]=[g,\mu]\otimes\Phi_r$. From this it follows that $\partial_{\mu}\Phi_r$ lies in isotypic component corresponding to the adjoint representation of $\mathfrak{g}$ with $\mathfrak{g}$-action defined on $\partial_{\mu}\Phi_r$ as
$$g\cdot\partial_{\mu}\Phi_r=\partial_{[g,\mu]}\Phi_r.$$
Moreover, the following fact is true.

\begin{predl}\label{th:SgAdjoint} \cite{K}
(i) $Hom_{\mathfrak{g}}\left(\mathfrak{g},S\left(\mathfrak{g}\right)\right)$ is a free $S\left(\mathfrak{g}\right)^{\mathfrak{g}}$-module.\\
(ii) This module is generated by $\tau_r\in Hom_{\mathfrak{g}}\left(\mathfrak{g},S\left(\mathfrak{g}\right)\right)$:
\begin{equation}
\tau_r:\mathfrak{g}\simeq\mathfrak{g}\otimes\Phi_r\xrightarrow{\tau}S\left(\mathfrak{g}\right)
\end{equation}
for $r=1,\ldots,n=\text{rk}\ \mathfrak{g}$.
\end{predl}

Similar results hold for $U\left(\mathfrak{g}\right)$. Note that $U\left(\mathfrak{g}\right)$ is isomorphic to $S\left(\mathfrak{g}\right)$ as a filtered $\mathfrak{g}$-module via the symmetrization (PBW) map $S(\mathfrak{g})\to U(\mathfrak{g})$. For any $x\in S\left(\mathfrak{g}\right)$ we denote by $\hat{x}$ its image in $U\left(\mathfrak{g}\right)$ under the symmetrization map. Let $ZU(\mathfrak{g})=U(\mathfrak{g})^{\mathfrak{g}}$ be the center of the universal enveloping algebra $U(\mathfrak{g})$, then $U(\mathfrak{g})$ is a free $ZU(\mathfrak{g})$-module described as follows.

\begin{predl}\label{th:UgAdjoint}
\cite{K} (i) $Hom_{\mathfrak{g}}\left(\mathfrak{g},U\left(\mathfrak{g}\right)\right)$ is a free $ZU\left(\mathfrak{g}\right)$-module.\\
(ii) This module is generated by the symmetrizations of all $\tau_r$: $\widehat{\tau_r}\in Hom_{\mathfrak{g}}\left(\mathfrak{g},U\left(\mathfrak{g}\right)\right)$ for $r=1,\ldots,n=\textrm{rk}\ \mathfrak{g}$.
\end{predl}


We use the above general theorems to describe explicitly some natural subspace in the (quantum) shift of argument subalgebra.

\begin{lemma}\label{le:AmuGenerators}
The quantum shift of argument subalgebra $\mathcal{A}_\mu$ contains the center $ZU(\mathfrak{g})$ and the elements $\widehat{\partial_{\mu}\Phi_r}$ for $r=1,\ldots,n$.
\end{lemma}

\begin{proof} The associated graded of $\widehat{A}$ in $S(\mathfrak{g}\oplus\mathfrak{g})$ can be regarded as a subalgebra in the algebra of polynomial functions in $x,y\in\mathfrak{g}$. According to \cite{R} it is generated by the coefficients at the powers of $t$ in $\Phi_r(x+ty)$ for all generators $\Phi_r\in S(\mathfrak{g})^{\mathfrak{g}}$. So by Theorem~\ref{th:SgAdjoint} $gr \widehat{A}$ contains $(S\left(\mathfrak{g}\right)\otimes1)^{\mathfrak{g}}$ and $(S\left(\mathfrak{g}\right)\otimes\mathfrak{g})^{\mathfrak{g}}$. Hence both $(U\left(\mathfrak{g}\right)\otimes1)^{\mathfrak{g}}$ and $(U\left(\mathfrak{g}\right)\otimes\mathfrak{g})^{\mathfrak{g}}$ belong to the universal shift of argument subalgebra $\widehat{A}\subset U(\mathfrak{g})\otimes S(\mathfrak{g})$. On the other hand by Theorem~\ref{th:UgAdjoint} we have surjective maps $$ev_{\mu}:[U(\mathfrak{g})\otimes1]^{\mathfrak{g}}\rightarrow ZU(\mathfrak{g}),$$ 
\begin{center}
and
\end{center}
$$ev_{\mu}:[U(\mathfrak{g})\otimes\mathfrak{g}]^{\mathfrak{g}}\rightarrow \bigoplus\limits_{r=1}^n ZU(\mathfrak{g})\widehat{\partial_{\mu}\Phi_r},$$ 
hence the assertion. \end{proof}

\begin{prop}\label{pr:Isotypic} (i) For any $\mathfrak{g}$-module homomorphism $\tau:\mathfrak{g}\to S\left(\mathfrak{g}\right)$ we have $\tau(\mu)\in A_\mu$. 

(ii) For any $\mathfrak{g}$-module homomorphism $\widehat{\tau}:\mathfrak{g}\to U\left(\mathfrak{g}\right)$ we have $\widehat{\tau}(\mu)\in \mathcal{A}_\mu$.
\end{prop}

\begin{proof} Clearly, for any $\mu\in\mathfrak{g}$ we have $\tau_r(\mu)\in A_\mu$ and by Lemma~\ref{le:AmuGenerators} we have $\widehat{\tau_r}(\mu)\in \mathcal{A}_\mu$. Since $S\left(\mathfrak{g}\right)^{\mathfrak{g}}$ is contained in $A_\mu$ and $ZU\left(\mathfrak{g}\right)$ is contained in $\mathcal{A}_\mu$, by Theorems~\ref{th:SgAdjoint}~and~\ref{th:UgAdjoint} we get the assertion.
\end{proof}

\subsection{Image of $\mathcal{B}^\mp$ in the centralizer algebra $U(\mathfrak{g}_{n})^{\mathfrak{g}_{n-1}}$.}

Recall the free generators $\sigma_2^{(2m)}$ and $s_{11}^{(2m+1)}$ of the Bethe subalgebra $\mathcal{B}^\mp\subset \textrm{Y}^{\mp}(2)$.

\begin{lemma}\label{le:ImageOfB}
(i) The images of the generators $\sigma_2^{(2m)}$ under $\varphi_n$ lie in the subalgebra generated by the centers of $U(\mathfrak{g}_n)$ and $U(\mathfrak{g}_{n-1})$.

(ii) The images of the generators $s_{11}^{(2m+1)}$ under $\varphi_n$ lie in isotypic component of $U(\mathfrak{g}_n)$ corresponding to the adjoint representation of $\mathfrak{g}_n$. More precisely, we have $\varphi_n(s_{11}^{(2m+1)})=\hat{\tau}(\mathcal{F}_{nn})$ for some $\mathfrak{g}_n$-homomorphism $\hat{\tau}:\mathfrak{g}_n\to U(\mathfrak{g}_n)$.
\end{lemma}

\begin{proof} (i) According to Molev and Olshanski \cite{MO} the centralizer subalgebra $U(\mathfrak{g}_n)^{\mathfrak{g}_{n-1}}$ is generated by $\varphi_n(\textrm{Y}^{\mp}(2))$ and the center of $U(\mathfrak{g}_n)$. Hence the elements $\varphi_n(\sigma_2^{(2m)})$ belong to the center of $U(\mathfrak{g}_n)^{\mathfrak{g}_{n-1}}$. On the other hand, the center of $U(\mathfrak{g}_n)^{\mathfrak{g}_{n-1}}$ is generated by the centers of $U(\mathfrak{g}_n)$ and $U(\mathfrak{g}_{n-1})$ (see e.g. \cite{Kn}).

(ii) From the commutation relations ($\ref{scruch3}$) of $\textrm{Y}^{\mp}(N)$ we obtain that
\begin{equation}\label{m1}
[s_{ij}^{(1)}, s_{kl}(u)]=\delta_{il}s_{kj}(u)-\delta_{kj}s_{il}(u)-\theta_{i,-l}\delta_{k,-i}s_{-l,j}(u)+\theta_{k,-j}\delta_{-j,l}s_{i,-k}(u)
\end{equation}
\begin{center}
and
\end{center}
\begin{equation}\label{m2}
[s_{ij}^{(1)}, s_{kl}^{(M)}]=\delta_{il}s_{kj}^{(M)}-\delta_{kj}s_{il}^{(M)}-\theta_{i,-l}\delta_{k,-i}s_{-l,j}^{(M)}+\theta_{k,-j}\delta_{-j,l}s_{i,-k}^{(M)}
\end{equation}
for fixed $M\in\mathbb{Z}_{\geq0}$.


Note that $\varphi_n(s_{ij}^{(1)})=F_{ij}$ hence by (\ref{m2}), the image of $s_{11}^{(2m+1)}$ lies in isotypic component of the adjoint representation of $\mathfrak{g}$.
\end{proof}

We consider the limit shift of argument subalgebra $\lim\limits_{\varepsilon\to0}\mathcal{A}_{\mu(\varepsilon)}$ where $\mu(\varepsilon)=F_{nn}+F_{n-1,n-1}\varepsilon+\ldots+F_{11}\varepsilon^{n-1}$. 

\begin{prop}\label{pr:BinA} The image of $\mathcal{B}^{\mp}$ in $U(\mathfrak{g}_k)^{\mathfrak{g}_{k-1}}$ lies in the subalgebra $\lim\limits_{\varepsilon\to0}\mathcal{A}_{\mu(\varepsilon)}\subset U(\mathfrak{g}_n)$.
\end{prop}

\begin{proof} By Theorem~\ref{th:quantShuvalov} the subalgebra $\lim\limits_{\varepsilon\to0}\mathcal{A}_{\mu(\varepsilon)}\subset U(\mathfrak{g}_n)$ is generated by the subalgebras $\mathcal{A}_{F_{kk}}$ in $U(\mathfrak{g}_k)$. By Lemma~\ref{le:ImageOfB} and Proposition~\ref{pr:Isotypic} the images of the generators of $\mathcal{B}^{\mp}$ in $U(\mathfrak{g}_k)^{\mathfrak{g}_{k-1}}$ belong to $\mathcal{A}_{F_{kk}}$ in $U(\mathfrak{g}_k)$.
\end{proof}

So for proving Theorem~A it remains to show that the subalgebra $\mathcal{A}^\mp$ in $U(\mathfrak{g}_n)$ generated by the centers of all $U(\mathfrak{g}_k)$ and by the images of $\mathcal{B}^{\mp}$ in all $U(\mathfrak{g}_k)^{\mathfrak{g}_{k-1}}$ has the same size as $\lim\limits_{\varepsilon\to0}\mathcal{A}_{\mu(\varepsilon)}$, i.e. has algebraically independent generators of the same degrees as $\mathcal{A}_\mu$ with generic $\mu$ has.

\begin{rem} \emph{In \cite{MY}, Example~5.6, Molev and Yakimova describe the subalgebra $\lim\limits_{\varepsilon\to0}\mathcal{A}_{\mu(\varepsilon)}$ for $\mathfrak{g}_n=\mathfrak{sp}_{2n}$ as the subalgebra generated by the centers of $U(\mathfrak{g}_k)$ and by $\hat{\tau}(F_{kk})$ for all the homomorphisms $\hat{\tau}:\mathfrak{g}_k\to U(\mathfrak{g}_k)$ as above. In the next subsection we show that $\mathcal{A}^-$ is generated by the same elements. Moreover, we give a similar description for the case $\mathfrak{g}_n=\mathfrak{o}_{2n+1}$.} 
\end{rem}

\subsection{Generators of the limit subalgebra}

Let $F^{(k)}$ ($k=1,...,n$) be the submatrix of $F$ (see (\ref{F-matrix})) lying in the intersection of rows and columns with $|i|,|j|\leq k$. Alternatively, in symplectic case one can treat $F^{(k)}$ as $F$-matrix for $\mathfrak{sp}_{2k}$ contained in $\mathfrak{sp}_{2n}$ as the span of $F_{ij}$ with $i,j=-k,...,-1,1,...,k$, and in odd orthogonal case as $F$-matrix for $\mathfrak{o}_{2k+1}$ contained in $\mathfrak{o}_{2n+1}$ as the span of $F_{ij}$ with $i,j=-k,...,-1,0,1,...,k$. the Poisson center of $S(\mathfrak{g}_k)$ is generated by the traces of even powers of $F^{(k)}$, i.e. by the elements $\Phi_m^{(k)}:=\textrm{Tr}(F^{(k)})^{2m}$ with $m=1,\ldots,k$. Similarly, we define $\mathcal{F}^{(k)}$ ($k=1,...,n$) be the submatrix of $\mathcal{F}$ lying in the intersection of rows and columns with $|i|,|j|\leq k$. The generators of the center of $U(\mathfrak{g}_k)\subset U(\mathfrak{g}_n)$ is generated by the elements $S_m^{(k)}:=\textrm{Tr}\ (\mathcal{F}^{(k)})^{2m}$ with $m=1,\ldots,k$. Clearly we have $\textrm{gr}\ S_m^{(k)}=\Phi_m^{(k)}$.

\begin{prop}\label{alg_ind_lemma}
The elements $\varphi_k(s_{11}^{(2m-1)}),S_m^{(k)}$ in $U(\mathfrak{g}_{k})^{\mathfrak{g}_{k-1}}$ for $k=1,\ldots,n$ and $m=1,\ldots,k$, are algebraically independent elements of $U(\mathfrak{g}_{n})$. 

\end{prop}

\begin{proof} The filtration on $\mathrm{Y}(2)$ given by $\deg(t_{ij}^{(m)})=m$ induces filtrations on both $\mathrm{Y}^-(2)$ and $\mathrm{Y}^+(2)$. The associated graded of $\mathrm{Y}^\mp(2)$ with respect to this filtration is a commutative algebra. The homomorphisms $\varphi_k:\mathrm{Y}^\mp(2)\to U(\mathfrak{g}_k)$ are compatible with this filtration on $\mathrm{Y}^\mp(2)$ and the PBW filtration on $U(\mathfrak{g}_k)$. Therefore to prove the Lemma it is sufficient to prove that $\textrm{gr}\ \varphi_k(s_{11}^{(2m-1)})$ and $\textrm{gr}\ \varphi_k(\sigma_2^{(2m)})$ under $\textrm{gr}\ \varphi_k$ are algebraically independent elements of $S(\mathfrak{g}_{n})=\textrm{gr}\ U(\mathfrak{g}_n)$. Throughout the proof of this Proposition we omit ``$\textrm{gr}$'' in the formulas and suppose that everything is in the associated graded algebra.

Consider the symplectic case first (the orthogonal case will be similar). 


From the definition of the homomorphism $\varphi_k$ in Theorem~\ref{th:centralizer} (i) we conclude that $\varphi_k(s_{ij}^{(m)})$ is \begin{equation}\label{images-of sij}\varphi_k(s_{ij}^{(m)})=\left[(\mathcal{F}^{(k)})^m\right]_{ij} + \sum\limits_{r=1}^{m}\chi_{k,r}\left[(\mathcal{F}^{(k)})^{m-r}\right]_{ij}
\end{equation}
where $\left[(\mathcal{F}^{(k)})^m\right]_{ij}$ stands for the $(i,j)$-th entry of the matrix $(\mathcal{F}^{(k)})^m$ and $\chi_{k,r}$ are some central elements of degree $r$. 

So we just need to prove that the following elements of $S(\mathfrak{sp}_{2n})$ are algebraically independent:
\begin{equation}
x_{km}:=\Phi_m^{(k)},\;y_{km}=\left[(F^{(k)})^m\right]_{kk},
\end{equation}
where $k=1,\ldots,n$ and $m=1,\ldots,k$.


Elements of $S(\mathfrak{g})$ are algebraically independent if their differentials at some point of $\xi\in\mathfrak{g}^*$ are linearly independent. We take the following principal nilpotent element of $\mathfrak{sp}_{2n}$ as such $\xi$. 

\begin{equation}
\tilde{e}:=\sum\limits_{i=1}^{n-1}1\otimes E_{i,i+1}+\sum\limits_{i=-n}^{-2}\left(-1\right)\otimes E_{i,i+1}+E_{-1,1}=\\\left(\begin{array}{cccccccccccc} 0 & 0 & 0 & ... & 0 & 0 & 0 & & ... & & & 0\\ 1 & 0 & 0 & ... & 0 & 0 & 0 & & ... & & & 0\\
0 & 1 & 0 & ... & 0 & 0 & 0 & & ... & & & 0\\ 0 & 0 & 1 & ... & 0 & 0 & 0 & & ... & & & 0\\ 
\vdots & \vdots & \vdots & \ddots & \vdots & \vdots & & & \ddots & & & \vdots\\ 0 & & ... & & 1 & 0 & 0 & & ... & & & 0\\
0 & & ... & & 0 & 1 & 0 & & 
... & & & 0\\ 0 & & ... & & 0 & 0 & -1 & 0 & & ... & & 0\\
0 & & ... & & 0 & 0 & 0 & -1 & 0 & ... & & 0\\ 0 & & ... & & 0 & 0 & 0 & 0 & -1 & 0 & ... & 0\\
\vdots & & & & \vdots & \vdots & \vdots & \vdots & \vdots & \ddots & & \vdots\\ 
0 & & & & & ... & & & & 0 & -1 & 0\\
\end{array}\right)
\end{equation}

Our goal is to show that the differentials of $x_{km}$, $y_{km}$ ($m=1,...,k$, $k=1,...,n$) at the principal nilpotent element are linearly independent. For this, we introduce a total order on set $\tilde{U}=\{dF_{ij}|j=i,...,1,-1,...,-i\;\textrm{and}\;i=n,...,1\}$:
$$dF_{ij}\succ dF_{i'j'}\;\;\textrm{if}\;i>i'\;\textrm{or}\;i=i'\;\textrm{and}\;j<j'.$$

We extend this order to a partial order on the set of all differentials by simply setting each $dF_{ij}$ not from $\tilde{U}$ to be less than any differential of set $\tilde{U}$.

\begin{lemma}
The differentials of the elements $x_{km}$, $y_{km}$ at $\tilde{e}$ have the following form (here $\sim$ denotes proportionality):
$$dx_{nn}|_{F=\tilde{e}}\sim dF_{n,-n}+(\textrm{linear\;combination\;of\;terms\;lower\;than\;}dF_{n,-n})$$
$$dy_{nn}|_{F=\tilde{e}}\sim dF_{n,-n+1}+(\textrm{linear\;combination\;of\;terms\;lower\;than\;}dF_{n,-n+1})$$
$$dx_{n,n-1}|_{F=\tilde{e}}\sim dF_{n,-n+2}+(\textrm{linear\;combination\;of\;terms\;lower\;than\;}dF_{n,-n+2})$$
\begin{equation}\label{upper_triangular_matrix_of_differentials}
dy_{n,n-1}|_{F=\tilde{e}}\sim dF_{n,-n+3}+(\textrm{linear\;combination\;of\;terms\;lower\;than\;}dF_{n,-n+3})
\end{equation}
$$...$$
$$dx_{n1}|_{F=\tilde{e}}\sim dF_{n,n-1}+(\textrm{linear\;combination\;of\;terms\;lower\;than\;}dF_{n,n-1})$$
$$dy_{n1}|_{F=\tilde{e}}\sim dF_{n,n}+(\textrm{linear\;combination\;of\;terms\;lower\;than\;}dF_{n,n})$$
\end{lemma}

\begin{proof}


Note that $dx_{nn}|_{F=\tilde{e}}$ is the only differential among all $dx_{km}|_{F=\tilde{e}}$ and $dy_{km}|_{F=\tilde{e}}$ which has a nonzero coefficient at $dF_{n,-n}$. Indeed, the only monomial with a nonzero contribution to $dF_{n,-n}$ is 
$$F_{n,-n}\cdot F_{-n,-n+1}\cdot...\cdot F_{-2,-1}\cdot F_{-1,1}\cdot F_{1,2}\cdot...\cdot F_{n-1,n}.$$
In $S(\mathfrak{sp}_{2n})$ it gives a contribution of $2n(-1)^{n-1}dF_{n,-n}$ to $dx_{nn}|_{F=\tilde{e}}$. Therefore, $dF_{n,-n}$ appears with a nonzero coefficient only in $dx_{nn}|_{F=\tilde{e}}$.

Now we look at all $dF_{n,-n+2i}$ ($i\geq1$) with $-n+2i<0$. Among all $dx_{km}|_{F=\tilde{e}}$ and $dy_{km}|_{F=\tilde{e}}$ with nonzero coefficients at $dF_{n,-n+2i}$, the one having the smallest value of $k+m$ is $dx_{n,n-i}|_{F=\tilde{e}}$. Clearly, the one with the smallest $k+m$ must have $k=n$, otherwise, $F_{n,-n+2i}$ would not be an entry of $F^{(k)}$. Then $m$ has to be at least $n-i$ since a nonzero contribution to $dF_{n,-n+2i}$ arises only from monomials of degree not less than $2n-2i$. Note that there is only one suitable monomial, namely:
$$F_{n,-n+2i}\cdot F_{-n+2i,-n+2i+1}\cdot...\cdot F_{-2,-1}\cdot F_{-1,1}\cdot F_{1,2}\cdot...\cdot F_{n-1,n}.$$
Such monomial occurs several number of times and overall contributes $(2n-2i)(-1)^{n-1-2i}dF_{n,-n+2i}=(2n-2i)(-1)^{n-1}dF_{n,-n+2i}$ to $dx_{n,n-i}|_{F=\tilde{e}}$. Hence the coefficient at $dF_{n,-n+2i}$ is nonzero.

In the same manner we treat the cases of $dF_{n,n-2i+1}$ ($i\geq1$) with $n-2i+1>0$. Again, we note that, among all $dx_{km}|_{F=\tilde{e}}$ and $dy_{km}|_{F=\tilde{e}}$ having nonzero coefficients at $dF_{n,n-2i+1}$, the one with the smallest value of $k+m$ is $dx_{n,i}|_{F=\tilde{e}}$: indeed, it is necessary to have $k=n$ and $m\geq i$ for $dF_{n,n-2i+1}$ to appear in $dx_{km}$. The only monomial in $x_{n,i}$ which contributes to $dF_{n,n-2i+1}$ is
$$F_{n,n-2i+1}\cdot F_{n-2i+1,n-2i+2}\cdot...\cdot F_{n-1,n}.$$
Each occurrence of this monomial gives a term equal to $(-1)^{2i}dF_{n,n-2i+1}$ in the differential at $F=\tilde{e}$.

Similarly, we can prove that, among all $dx_{km}|_{F=\tilde{e}}$ and $dy_{km}|_{F=\tilde{e}}$ with nonzero coefficients at $dF_{n,-n+1}$, the one having the smallest value of $k+m$ is $dy_{nn}$. First we notice that $k$ and $m$ have to be equal to $n$. Then we state that nonzero contributions can be only obtained from the following monomial:
$$F_{n,-n+1}\cdot F_{-n+1,-n+2}\cdot...\cdot F_{-2,-1}\cdot F_{-1,1}\cdot F_{1,2}\cdot...\cdot F_{n-1,n}.$$

The proofs of the two following statements are identical to the already given ones:
\begin{itemize}
\item Let $-n+2i+1<0$. Then $dy_{n,n-i}$ has the smallest value of $k+m$ among all $dx_{km}|_{F=\tilde{e}}$ and $dy_{km}|_{F=\tilde{e}}$ with nonzero coefficients at $dF_{n,-n+2i+1}$.
\item Let $n-2i+2>0$. Then $dy_{ni}$ has the smallest value of $k+m$ among all $dx_{km}|_{F=\tilde{e}}$ and $dy_{km}|_{F=\tilde{e}}$ with nonzero coefficients at $dF_{n,n-2i+2}$.
\end{itemize}
\end{proof}

In case of $\mathfrak{o}_{2n+1}$ the principal nilpotent element is
\begin{equation}
\tilde{e}=\sum\limits_{i=-1}^{n-1}1\otimes E_{i,i+1}+\sum\limits_{i=-n}^{-2}\left(-1\right)\otimes E_{i,i+1}.
\end{equation}
Doing similar steps as in the symplectic case we arrive at:
$$dx_{km}|_{F=\tilde{e}}\sim dF_{k,k-2m+1}+(\textrm{linear\;combination\;of\;terms\;lower\;than\;}dF_{k,k-2m+1})$$
\begin{equation}\label{upper_triangular_matrix_of_differentials}
dy_{km}|_{F=\tilde{e}}\sim dF_{k,k-2m+2}+(\textrm{linear\;combination\;of\;terms\;lower\;than\;}dF_{k,k-2m+2})
\end{equation}

Again by combining the above equations for $k=n,n-1,\ldots,1$ we prove the linear independence of all $dx_{km}|_{F=\tilde{e}}$ and $dy_{km}|_{F=\tilde{e}}$, hence algebraic independence of the set $x_{km},y_{km}$ with $k=1,\ldots,n$ and $m=1,\ldots,k$.
\end{proof}

Now we can proceed to Theorem A of this paper.

\textit{Proof of the Theorem A:} The Poincar\'e series for the sublagebra $\lim\limits_{\varepsilon\to0}\mathcal{A}_{\mu(\varepsilon)}$ coincides with that for $\mathcal{A}_\mu$ with generic $\mu\in\mathfrak{h}^{reg}$ and is equal to:
\begin{equation}
P_{\mathcal{A}_{\mu(\varepsilon)}}(x)=\prod\limits_{k=1}^n\frac{1}{(1-x^{2k-1})^{n-k+1}(1-x^{2k})^{n-k+1}}.
\end{equation}
On the other hand the elements $\varphi_k(s_{11}^{(2m-1)}),S_{k}^{(m)}$ are algebraically independent by Proposition~\ref{alg_ind_lemma} hence generate a commutative subalgebra $\mathcal{A}^{\mp}$  with the same Poincar\'e series. By Proposition~\ref{pr:BinA} we have $$\varphi_k(s_{11}^{(2m-1)})\ ,\ \varphi_k(\sigma_2^{(2m)})\in\lim\limits_{\varepsilon\to0}\mathcal{A}_{\mu(\varepsilon)},$$ and by Theorem~\ref{th:quantShuvalov} we have $S_k^{(m)}\in\lim\limits_{\varepsilon\to0}\mathcal{A}_{\mu(\varepsilon)}$, so Theorem~A follows.  $\Box$

\section{Proof of theorem B.}

\subsection{Asymptotics of the eigenbasis.}

Let $\alpha_1,\beta_1,\ldots,\alpha_k,\beta_k$ be fixed numbers such that $\alpha_i-\beta_i$ are nonnegative integers. Consider the following representation of $\textrm{Y}(2)$:
\begin{equation}
L(z_1,...,z_k)=L(z_1+\alpha_1,z_1+\beta_1)\otimes...\otimes L(z_k+\alpha_k,z_k+\beta_k).
\end{equation}
where $z_i\in\mathbb{C}$ are free parameters.

$L(z_1,...,z_k)$ is naturally a $\mathrm{Y}^-(2)$-module, so the commutative subalgebra $\mathcal{B}^-$ acts on $L(z_1,...,z_k)$. We can make the space $L(z_1,...,z_k)$ independent on the $z_i$'s by identifying it with the tensor product $L=V_{\alpha_1-\beta_1}\otimes...\otimes V_{\alpha_k-\beta_k}$ as the $U(\mathfrak{sl}_2)^{\otimes k}$-module. Then the image of $\mathcal{B}^-$ in $\text{End}(L)$ depends on $z_i$, we denote it by $\mathcal{B}^-(z_1,...,z_k)$. Let us describe the limit of $\mathcal{B}^-(z_1,...,z_k)$ as $z_i\to\infty$: 


\begin{lemma}\label{le:assymptotic_sp}
Suppose that $z_i=tu_i$ ($i=1,...,k$) and $u_i\ne u_j$. Then the limit of the commutative subalgebra $lim_{t\rightarrow\infty}\mathcal{B}^-(tu_1,...,tu_k)$ is generated by $h^{(i)}=1^{\otimes(i-1)}\otimes h\otimes1^{\otimes(k-i)}\in U(\mathfrak{sl}_2)^{\otimes k}$, for all $i=1,\ldots,k$.
\end{lemma}

\begin{proof} First, we compute the image of $s_{11}(u)$ in $U(\mathfrak{gl}_2)^{\otimes k}$:
$$\Delta^{k-1}(s_{11}(u))=\Delta^{k-1}\left(t_{11}(u)t_{-1,-1}(-u)-t_{1,-1}(u)t_{-1,1}(-u)\right)=$$ $$=\Delta^{k-1}\left(t_{11}(u)\right)\Delta^{k-1}\left(t_{-1,-1}(-u)\right)-\Delta^{k-1}\left(t_{1,-1}(u)\right)\Delta^{k-1}\left(t_{-1,1}(-u)\right)=$$ $$=\sum\limits_{\substack{i_1,i_2,...,i_{k-1}\in\{-1,1\}\\j_1,j_2,...,j_{k-1}\in\{-1,1\}}}t_{1,i_1}(u)t_{-1,j_1}(-u)\otimes t_{i_1,i_2}(u)t_{j_1,j_2}(-u)\otimes...\otimes t_{i_{k-1},1}(u)t_{j_{k-1},-1}(-u)-$$ $$-\sum\limits_{\substack{i_1,i_2,...,i_{k-1}\in\{-1,1\}\\j_1,j_2,...,j_{k-1}\in\{-1,1\}}}t_{1,i_1}(u)t_{-1,j_1}(-u)\otimes t_{i_1,i_2}(u)t_{j_1,j_2}(-u)\otimes...\otimes t_{i_{k-1},-1}(u)t_{j_{k-1},1}(-u)\mapsto$$ \begin{equation}\label{term1}\sum\limits_{\substack{i_1,i_2,...,i_{k-1}\in\{-1,1\}\\j_1,j_2,...,j_{k-1}\in\{-1,1\}}}\left(\delta_{1,i_1}+E_{1,i_1}u^{-1}\right)\left(\delta_{-1,j_1}-E_{-1,j_1}u^{-1}\right)\otimes...\otimes\left(\delta_{i_{k-1},1}+E_{i_{k-1},1}u^{-1}\right)\left(\delta_{j_{k-1},-1}-E_{j_{k-1},-1}u^{-1}\right)-\end{equation} \begin{equation}\label{term2}\sum\limits_{\substack{i_1,i_2,...,i_{k-1}\in\{-1,1\}\\j_1,j_2,...,j_{k-1}\in\{-1,1\}}}\left(\delta_{1,i_1}+E_{1,i_1}u^{-1}\right)\left(\delta_{-1,j_1}-E_{-1,j_1}u^{-1}\right)\otimes...\otimes\left(\delta_{i_{k-1},-1}+E_{i_{k-1},-1}u^{-1}\right)\left(\delta_{j_{k-1},1}-E_{j_{k-1},1}u^{-1}\right).\end{equation}

Note that $E_{11}$ and $E_{-1,-1}$ act on $L\left(z_i+\alpha_i,z_i+\beta_i\right)$ as $z_i+\frac{\alpha_i+\beta_i}{2}+\frac{h}{2}$ and $z_i+\frac{\alpha_i+\beta_i}{2}-\frac{h}{2}$ respectively, where $h$ is the standard generator of $\mathfrak{sl}_2$ which acts as $\textrm{diag}(\alpha_i-\beta_i,\alpha_i-\beta_i-2,...,-\alpha_i+\beta_i+2,-\alpha_i+\beta_i)$. The element $E_{1,-1}$ acts simply as the $\mathfrak{sl}_2$ generator $e$ and $E_{-1,1}$ acts as $f$.

We want to find the image of each $s_{11}^{(2m+1)}$ ($m\in\mathbb{Z}_{\geq0}$) regarded now as an element of $\mathbb{C}\left[z_1,...,z_k\right]\otimes U(\mathfrak{sl}_2)^{\otimes k}$, i.e. express it in terms of the standard generators $e^{(i)}:=1^{\otimes(i-1)}\otimes e\otimes 1^{\otimes(k-i)}$, $h^{(i)}:=1^{\otimes(i-1)}\otimes h\otimes 1^{\otimes(k-i)}$, $f^{(i)}:=1^{\otimes(i-1)}\otimes f\otimes 1^{\otimes(k-i)}$ and the parameters $z_i$. The limit $t\to\infty$ depends on the leading term of this expression, i.e. on the highest degree component in the variables $z_i$.

The image of $s_{11}^{(2m+1)}$ ($m\in\mathbb{Z}_{\geq0}$) in $\mathbb{C}\left[z_1,...,z_k\right]\otimes U(\mathfrak{sl}_2)^{\otimes k}$ cannot have monomials of the degree greater than $2m$. Indeed, the expression of this image in the generators of $\mathbb{C}\left[z_1,...,z_k\right]\otimes U(\mathfrak{sl}_2)^{\otimes k}$ is the sum of two terms given by~(\ref{term1})~and~(\ref{term2}). Consider the term~(\ref{term1}). Suppose that at least one of the $i_1,i_2,...i_{k-1}$ is equal to -1 or at least one of $j_1,j_2,...,j_{k-1}$ is equal to 1. Then among the pairs $(1,i_1)$, $(i_1,i_2)$, ..., $(i_{k-1},1)$, $(-1,j_1)$, $(j_1,j_2)$, ..., $(j_{k-1},-1)$ there are at least two consisting of different elements. In such case the factors of the summand corresponding to the choice of $I=(i_1,i_2,...,i_{k-1})$ and $J=(j_1,j_2,...,j_{k-1})$ will contribute to the degree of $u^{-1}$ but not to the degree in $z_i$. Since every $u^{-1}$ increases the degree of the polynomial component by at most 1 it follows that under our assumption we cannot get a polynomial of the degree more than $2m-1$. Now suppose we have $I=(1,1,...,1)$ and $J=(-1,-1,...,-1)$. The corresponding summand is:
\begin{equation}\label{good_choice}
\left(1+hu^{-1}+\left(z_1^2-h^2/4\right)u^{-2}\right)\otimes\left(1+hu^{-1}+\left(z_2^2-h^2/4\right)u^{-2}\right)\otimes...\otimes\left(1+hu^{-1}+\left(z_k^2-h^2/4\right)u^{-2}\right)
\end{equation}
Clearly, if we need to get $u^{-2m-1}$ in one of the tensor terms we have to choose $hu^{-1}$ which bounds the maximal possible degree of the coefficient with $2m$.

In the term~(\ref{term2}) we have at least two pairs among $(1,i_1)$, $(i_1,i_2)$, ..., $(i_{k-1},-1)$, $(-1,j_1)$, $(j_1,j_2)$, ..., $(j_{k-1},1)$ consisting of different elements, independently of the choice of $I=(i_1,i_2,...,i_{k-1})$ and $J=(j_1,j_2,...,j_{k-1})$. So the above argument guarantees that all possible monomials in~(\ref{term2}) have the degree not exceeding $2m-1$.

On the other hand, the degree of $2m$ in $z_i$ can be obtained only in the case described above (\ref{good_choice}) and the corresponding term can be explicitly written as:
\begin{equation}\label{spectrum_entries}
\sum_{i=1}^ke_m(z_1^2,...,\widehat{z_i^2},...,z_k^2)\otimes h^{(i)},
\end{equation}
where $e_m(z_1^2,...,\widehat{z_i^2},...,z_k^2)$ denotes elementary symmetric polynomial of degree $m$ in variables $z_j^2$ with $j\neq i$. Here $h^{(i)}$ as usual stands for $1^{\otimes(i-1)}\otimes h\otimes1^{\otimes(n-i)}\in U(\mathfrak{sl}_2)^{\otimes k}$.

Let us slightly modify the non-trivial generators $s_{11}^{(2m+1)}$ ($m\in\mathbb{Z}_{\geq0}$) of $\mathcal{B}^-$:
\begin{equation}
s_{11}^{(2m+1)}\rightsquigarrow \widetilde{s_{11}^{(2m+1)}}(t):=t^{-2m}s_{11}^{(2m+1)}
\end{equation}
for $t\in\mathbb{C}$ and recall that $z_i=tu_i$.

The limits of the images of the new generators $\widetilde{s_{11}^{(2m+1)}}(t)$ ($m\in\mathbb{Z}_{\geq0}$) as $t\to\infty$ are:
\begin{equation}\label{asympt_spectrum}
\lim\limits_{t\to\infty}\widetilde{s_{11}^{(2m+1)}}(t)=\sum_{i=1}^ke_m(u_1^2,...,\widehat{u_i^2},...,u_k^2)\otimes h^{(i)}
\end{equation}

Due to the definition of the $\mathrm{Y}^-(2)$-module $L$ the elements $s_{11}^{(2m+1)}$ with $m>k-1$ act by zero on $L$, so $s_{11}^{(2m+1)}(t)$ with $m>k-1$ act by zero as well. From (\ref{asympt_spectrum}) it follows that the $t\to\infty$ limit of the subalgebra generated by the images of $s_{11}^{(2m+1)}(t)$ with $m\le k-1$  is generated by $\{h^{(i)}|i=1,...,k\}$. Indeed, we just need to show that the following matrix is non-degenerate:
\begin{equation}\label{vandermonde_kind_of_matrix}
\left(e_j(u_1^2,...,\widehat{u_i^2},...,u_k^2)\right)_{i,j=1}^k.
\end{equation}
First we notice that its determinant is a skew-symmetric polynomial in $(u_1^2,u_2^2,...,u_k^2)$. Then we look at the coefficient before highest monomial $u_1^{2k-2}\cdot u_2^{2k-4}\cdot....\cdot u_{k-1}^2$. Such monomial in the above determinant appears just once, namely, as the product of the diagonal elements. Therefore, the determinant of our matrix coincides with the Vandermond polynomial which proves the non-degeneracy of the matrix. Hence the limit $\lim\limits_{t\rightarrow\infty}\mathcal{B}^{-}(tu_1,...,tu_k)$ is the subalgebra generated by $h^{(i)}\in U(\mathfrak{sl}_2)^{\otimes k}$ ($i=1,...,k$). 
\end{proof}

Now let us turn to the case of $\mathfrak{o}_{2n+1}$. Consider the following $\mathrm{Y}^+(2)$-module:
\begin{equation}
L'(z_1,...,z_k,\delta)=L(z_1,...,z_k)\otimes W(\delta),
\end{equation}
where $W(\delta)$ ($\delta\in\mathbb{C}$) is one-dimensional representation of $\mathrm{Y}^+(2)$ introduced earlier in Section 1. Then asymptotically the spectrum of $L'(z_1,...,z_n,\delta)$ is simple due to the following lemma.

\begin{lemma}\label{le:assymptotic_o}
Suppose that $z_i=tu_i$ ($i=1,...,k$) and $u_i\ne u_j$. Then the limit of the commutative subalgebra $lim_{t\rightarrow\infty}\mathcal{B}^+(tu_1,...,tu_k)$ is generated by $h^{(i)}+(\delta-\frac{1}{2})\cdot id\in U(\mathfrak{sl}_2)^{\otimes k}\otimes\mathbb{C}[\delta]$, for all $i=1,\ldots,k$.
\end{lemma}

\begin{proof} By using the formula for coproduct of $s_{ab}(u)$ ($|a|,|b|=1$):
\begin{equation}
\Delta\left(s_{ab}(u)\right)=\sum_{c,d\in\{-1,1\}}t_{ac}(u)t_{-b,-d}(-u)\otimes s_{cd}(u)
\end{equation}
we find that
$$\Delta^k\left(s_{11}(u)\right)=\sum\limits_{\substack{i_1,...,i_{k-1}\in\{-1,1\}\\j_1,...,j_{k-1}\in\{-1,1\}\\c,d\in\{-1,1\}}}t_{1,i_1}(u)t_{-1,-j_1}(-u)\otimes t_{i_1,i_2}(u)t_{j_1,j_2}(-u)\otimes...\otimes t_{i_{k-1},c}(u)t_{j_{k-1},-d}(-u)\otimes s_{cd}(u)=$$
$$=\sum\limits_{\substack{i_1,...,i_{k-1}\in\{-1,1\}\\j_1,...,j_{k-1}\in\{-1,1\}}}t_{1,i_1}(u)t_{-1,-j_1}(-u)\otimes t_{i_1,i_2}(u)t_{j_1,j_2}(-u)\otimes...\otimes t_{i_{k-1},1}(u)t_{j_{k-1},-1}(-u)\otimes s_{11}(u)+$$
$$+\sum\limits_{\substack{i_1,...,i_{k-1}\in\{-1,1\}\\j_1,...,j_{k-1}\in\{-1,1\}}}t_{1,i_1}(u)t_{-1,-j_1}(-u)\otimes t_{i_1,i_2}(u)t_{j_1,j_2}(-u)\otimes...\otimes t_{i_{k-1},-1}(u)t_{j_{k-1},1}(-u)\otimes s_{-1,-1}(u)$$
since $s_{1,-1}(u)$ and $s_{-1,1}(u)$ act as zeros on $W(\delta)$.

Therefore, by making similar observations as in the symplectic case we derive that the image of $s_{11}^{(2m+1)}$ in $U\left(\mathfrak{sl}_2\right)^{\otimes k}\otimes\mathbb{C}[\delta]$ has the following term with the highest degree polynomial coefficient:
\begin{equation}
\sum_{i=1}^ke_m(z_1^2,...,\widehat{z_i^2},...,z_k^2)\otimes\left(h^{(i)}+\left(\delta-\frac{1}{2}\right)\cdot id\right),
\end{equation}
where $e_m(z_1^2,...,\widehat{z_i^2},...,z_k^2)$ are the same symmetric polynomials as before. Hence we conclude that the subalgebra $\lim\limits_{t\rightarrow\infty}\mathcal{B}^+(tu_1,...,tu_k)$ coincides with the subalgebra generated by $\{h^{(i)}+\left(\delta-\frac{1}{2}\right)\cdot id|i=1,...,k\}$. 
\end{proof}

\begin{cor} The subalgebra $\mathcal{B}^-$ (resp. $\mathcal{B}^+$) acts on $L(z_1,\ldots,z_k)$ (resp. $L'(z_1,\ldots,z_k,1)\oplus L'(z_1,\ldots,z_k,0)$ or $L'(z_1,\ldots,z_k,\frac{1}{2})\oplus L'(z_1-1,\ldots,z_k,\frac{1}{2})$) with simple spectrum for $z_i=tu_i$ with $t$ being large enough and $u_i\ne u_j$. 
\end{cor}

\begin{proof} The symplectic case immediately follows from Lemma \ref{le:assymptotic_sp}. 

In orthogonal case when $\lambda_i$ are integers by Lemma \ref{le:assymptotic_o} we notice that this subalgebra is generated by $h^{(i)}$ ($i=1,...,k$) and $h^{(1)}$ acts differently on the first components of $U$ and $U'$, therefore, the joint spectrum is simple. In the second case when $\lambda_i$ are half-integers we have subalgebra $\{h^{(i)}-\frac{1}{2}\cdot id|i=1,...,k\}$ acting on $U$ and subalgebra $\{h^{(i)}+\frac{1}{2}\cdot id|i=1,...,k\}$ acting on $U'$. Hence, the spectrum is simple.
\end{proof}

\subsection{Proof of the Theorem B.} Now we can index the spectrum of $\mathcal{A}^\mp$ in $V_\lambda$ by a combinatorial datum. 

We set $\alpha_1,\beta_1,\ldots,\alpha_k,\beta_k$ as in Theorems~\ref{th:GTsp}~and~\ref{th:GTo}. Note that the condition that $\mathcal{B}^\mp(z_1,\ldots,z_{k})$ has simple spectrum in $L$ is a Zariski-open condition on the parameters $z_i$. Since it is satisfied in the limit it is in fact satisfied for all $(z_1,\ldots,z_{k})$ belonging to some Zariski open dense subset $U_k\subset\mathbb{C}^{k}$. Moreover, according to \cite{KR}, Corollary~11.6, the spectrum of $\mathcal{A}_k=\varphi_k(\mathcal{B}^{\mp})$ on the multiplicity space $V_\lambda^\mu=L(\alpha_1,\beta_1)\otimes\ldots\otimes L(\alpha_k,\beta_k)$ is simple, hence the point $0=(0,\ldots,0)$ belongs to $U_k$. Since the real codimension of $U_k$ in $\mathbb{C}^{k}$ is $2$ we have a path $z(t)$ ($t\in[0,\infty)$) connecting $0$ with the limit described above, such that $z(t)\in U_k$ for all $t\in[0,\infty)$. For definiteness we can consider the complex line $\{z_l=(l-1)z|\ z\in\mathbb{C}\}$ in the space of parameters $\mathbb{C}^{k}$ and choose the path $z(t)$ going along the real half-line $[0,\infty)$ on this complex line and avoiding the ``bad'' points on it in the counter-clockwise direction. 

For any $t\in[0,\infty]$ the corresponding subalgebra $\mathcal{B}^\mp(z_1(t),\ldots,z_{k}(t))$ has simple spectrum in $L$ hence the corresponding eigenbasis is a finite cover of the segment $[0,\infty]$. So the parallel transport from $0$ to $\infty$ gives a bijection between the eigenbasis of $\mathcal{B}^\mp$ in $V_\lambda^\mu$ (at $z(0)=0$) and the eigenbasis at the limit $t\to\infty$ which is just the weight basis of $L$ as a $U(\mathfrak{sl}_2)^{\otimes k}$-module.

In the symplectic case this weight basis is just the product of the weight bases of the $\mathfrak{sl}_2$-modules $V_{\alpha_1-\beta_1},\ldots,V_{\alpha_k-\beta_k}$. It is naturally indexed by the collections of integers $\lambda_1',\ldots,\lambda_k'$ such that $\alpha_i\ge\lambda_i'\ge\beta_i$. This means that $$0\ge\lambda_1'\ge\max(\lambda_1,\mu_1)$$ $$\min(\lambda_1,\mu_1)\ge\lambda_2'\ge\max(\lambda_2,\mu_2)$$
$$\ldots$$ $$\min(\lambda_{k-1},\mu_{k-1})\ge\lambda_k'\ge\lambda_k.$$

Similarly, in the odd-dimensional orthogonal case the weight basis is indexed by the collections of $\sigma,\lambda_1',\ldots,\lambda_k'$ where $\sigma\in\{0,1\}$ indicates one of the two summands in the decomposition of the multiplicity space, $\lambda_i'\in\mathbb{Z}+\lambda_i$ indicates the $\mathfrak{sl}_2$-weight in the $i$-th tensor factor, and the following obvious inequalities hold: $$\min(0,-\sigma+\frac{1}{2})\ge\lambda_1'\ge\max(\lambda_1,\mu_1)$$ $$\min(\lambda_1,\mu_1)\ge\lambda_2'\ge\max(\lambda_2,\mu_2)$$
$$\ldots$$ $$\min(\lambda_{k-1},\mu_{k-1})\ge\lambda_k'\ge\lambda_k.$$

So in both cases of $\mathfrak{sp}_{2n}$ and $\mathfrak{o}_{2n+1}$ we get an indexing of the spectrum of $\mathcal{B}^\mp$ in the multiplicity space $V_\lambda^\mu$ by collections of numbers which form the intermediate row between $\lambda$ and $\mu$ in the (C or B type) Gelfand-Tsetlin table. Doing the above procedure for all $k=n,n-1,\ldots,1$ we get the indexing of the eigenbasis of $\mathcal{A}^\mp$ in $V_\lambda$ by the Gelfand-Tsetlin patterns of the corresponding type. $\Box$

\subsection{Speculation on definiteness of indexing and relation to crystals.} 

Note that the indexing of the eigenbasis of $\mathcal{A}^\mp$ in $V_\lambda$ depends on the choice of the path $z(t)$. We conjecture that it can be made independent of this choice if we put some reasonable restrictions on $z(t)$. Namely we can restrict to real values of the parameters $z_l$ such that $z_1>z_2>\ldots>z_{k}$. 

\begin{conj}\label{cj:SimpleSpec} Under the above assumptions on the parameters, the algebra $\mathcal{B}^-$ (resp. $\mathcal{B}^+$) acts on $L(z_1,\ldots,z_{k})$ (resp. $L'(z_1,\ldots,z_k,1)\oplus L'(z_1,\ldots,z_k,0)$ or $L'(z_1,\ldots,z_k,\frac{1}{2})\oplus L'(z_1-1,\ldots,z_k,\frac{1}{2})$) with simple spectrum.
\end{conj}

The obvious consequence of Conjecture~\ref{cj:SimpleSpec} is that for any path $z(t)$ in the real sector $z_1>z_2>\ldots>z_{k-1}$ the indexing of the eigenbasis of $\mathcal{A}^\mp$ in $V_\lambda$ is the same, since this sector is simply-connected. 

The set of Gelfand-Tsetlin patterns of classical type (with the upper row fixed to be $\lambda$) carries a structure of the normal crystal $B_\lambda$ due to Littelmann \cite{L}. On the other hand, the set of eigenlines of a shift of argument subalgebra in the representation $V_\lambda$ has a structure of the normal crystal $B_\lambda$ as well, due to \cite{KR}. 

\begin{conj}
These crystal structures are the same.
\end{conj}

\newpage
\addcontentsline{toc}{section}{References}
\bibliographystyle{halpha}
\bibliography{BC_case_bibliography}

\footnotesize{
{\bf L.R.}: National Research University
Higher School of Economics, Russian Federation,\\
Department of Mathematics, 6 Usacheva st, Moscow 119048;\\
Institute for Information Transmission Problems of RAS;\\
{\tt leo.rybnikov@gmail.com}}

\footnotesize{
{\bf M.Z.}: National Research University
Higher School of Economics, Russian Federation,\\
Department of Mathematics, 6 Usacheva st, Moscow 119048;\\
{\tt zavalin.academic@gmail.com}}

\end{document}